\documentclass[11pt]{article}
\pdfoutput=1

\usepackage[a4paper,margin=2.9cm]{geometry}
\usepackage{amsmath,amssymb,amsthm,amsfonts}
\usepackage{bm}
\usepackage{comment}
\usepackage{thmtools}
\usepackage{graphicx}
\usepackage{booktabs}
\usepackage{caption}
\usepackage{subcaption}
\usepackage{enumitem}
\usepackage[colorlinks=true,allcolors=blue]{hyperref}
\usepackage{cleveref}

\theoremstyle{plain}
\newtheorem{theorem}{Theorem}[section]
\newtheorem{proposition}[theorem]{Proposition}
\newtheorem{lemma}[theorem]{Lemma}

\theoremstyle{definition}
\newtheorem{assumption}[theorem]{Assumption}
\newtheorem{remark}[theorem]{Remark}

\crefname{theorem}{theorem}{theorems}
\Crefname{theorem}{Theorem}{Theorems} 
\crefname{lemma}{lemma}{lemmas} 
\Crefname{lemma}{Lemma}{Lemmas} 
\crefname{proposition}{proposition}{propositions} 
\Crefname{proposition}{Proposition}{Propositions} 
\crefname{corollary}{corollary}{corollaries} 
\Crefname{corollary}{Corollary}{Corollaries} 
\crefname{assumption}{assumption}{assumptions} 
\Crefname{assumption}{Assumption}{Assumptions}

\newcommand{\bu}{\bm{u}}
\newcommand{\bv}{\bm{v}}
\newcommand{\bw}{\bm{w}}
\newcommand{\bq}{\bm{q}}
\newcommand{\bff}{\bm{f}}
\newcommand{\bg}{\bm{g}}
\newcommand{\bn}{\bm{n}}
\newcommand{\bt}{\bm{t}}
\newcommand{\bx}{\bm{x}}

\newcommand{\bzeta}{\bm{\zeta}}
\newcommand{\bsig}{\bm{\sigma}}
\newcommand{\beps}{\bm{\varepsilon}}
\newcommand{\btau}{\bm{\tau}}
\newcommand{\bvt}{\bm{\vartheta}}
\newcommand{\bI}{\bm{\mathbb{I}}}
\newcommand{\bK}{K}
\newcommand{\bS}{S}
\newcommand{\bE}{\bm {\mathcal {E}}}
\newcommand{\bG}{\bm{G}}
\newcommand{\bL}{\bm{\mathcal{L}}}
\newcommand{\bhu}{\hat{\bu}}
\newcommand{\bhv}{\hat{\bv}}
\newcommand{\bhw}{\hat{\bw}}
\newcommand{\bhq}{\hat{\bq}}
\newcommand{\bp}{\bm{p}}
\newcommand{\bhp}{\hat{\bp}}
\newcommand{\bhe}{\hat{\bm{e}}}
\newcommand{\jump}[1]{[\![#1]\!]}
\newcommand{\Th}{\mathcal{T}_h}
\newcommand{\Fh}{\mathcal{F}_h}
\newcommand{\Pk}[1]{\mathbb{P}^{#1}}
\newcommand{\bPk}[1]{{\bm{P}}^{#1}}
\newcommand{\sPk}[1]{\bm{\mathcal{P}}^{#1}}
\newcommand{\Vh}{\hat{\bm{V}}_h}
\newcommand{\Vhz}{\hat{\bm{V}}_{h0}}
\newcommand{\VT}{\hat{\bm{V}}_T}
\newcommand{\tnorm}[1]{|\!|\!|#1|\!|\!|}
\newcommand{\CC}{\mathbb{C}}
\newcommand{\Iint}{\hat{\bm{I}}}
\newcommand{\dtt}{\partial_{tt}}
\newcommand{\dt}{\partial_{t}}
\DeclareMathOperator{\tr}{tr}
\DeclareMathOperator{\dive}{div}

\begin{document}

\title{\textbf{An Unfitted Hybrid High-Order Method for\\ the Elastodynamics Problem with Imperfect Interface}}
\author{Peiqi Huang\thanks{Department of Applied Mathematics, Nanjing Forestry University,
Nanjing 210037, People's Republic of China. E-mail: \texttt{pqhuang@njfu.edu.cn}.}
\ \ and Erik Burman\thanks{Department of Mathematics, University College London,
Gower Street, London WC1E 6BT, UK. E-mail: \texttt{e.burman@ucl.ac.uk} (corresponding author).}}
\date{}
\maketitle

\begin{abstract}
We design and analyse an unfitted hybrid high-order (HHO) method for the elastic wave
equation in a medium made of two components separated by an imperfect interface of linear
slip type, across which the traction is continuous and the displacement jump is
proportional to the traction through a compliancy tensor
$\bK=\alpha\bI+(\beta-\alpha)\bn\otimes\bn$. The mesh is not fitted to the interface: the
discrete unknowns are doubled in the cut cells, the small cuts are cured by a cell
agglomeration procedure, and no unknown is attached to the interface. The two specific
ingredients of the method are a local symmetric strain reconstruction in each cut subcell,
which incorporates the interface condition through the regularised interface stiffness
$\bS_h=(h_T\delta^{-1}\bI+\bK)^{-1}$ in the spirit of Hansbo and Hansbo {\em{A finite element method for the simulation of strong and weak
  discontinuities in solid mechanics.}}
{Comput. Methods Appl. Mech. Engrg.}, 193, 2004, and an interface
stabilisation built from the same matrix. A single formulation therefore covers the whole
range of compliancies, from the perfectly bonded interface $\alpha=\beta=0$, where the
interface stabilisation acts as a Nitsche penalty, to the traction-free interface obtained
as $\alpha,\beta\to+\infty$. For the space semi-discrete problem we prove that the discrete
bilinear form is coercive and continuous, and we derive an energy-error estimate of order
$h^{k+1}$ and an $L^2$-error estimate of order $h^{k+2}$, with constants independent of the
compliancy parameters and of how the interface cuts the mesh. The scheme is combined either
with the Newmark scheme, which conserves a discrete energy exactly, or with singly
diagonally implicit Runge--Kutta schemes of order up to four. Numerical experiments in two
dimensions confirm the predicted convergence rates for $k\in\{1,2,3\}$, the robustness with
respect to the compliancy over sixteen orders of magnitude, and illustrate the propagation
of elastic waves across an unresolved slipping interface.
\end{abstract}



\section{Introduction}\label{sec:intro}

Elastic waves propagate in composite materials, in glued or welded assemblies and in
geological media across a large number of material interfaces. When the two components are
perfectly bonded, both the displacement and the normal traction are continuous across the
interface. In many situations of practical interest the bonding is imperfect: a thin
adhesive layer, a damaged zone or a micro-cracked region between the two materials allows
the two sides to slide or to separate. A standard way of modelling such an interface without
resolving the thin layer is to keep the traction continuous but to let the displacement jump
be proportional to it,
\begin{equation}\label{eq:slip}
\jump{\bsig(\bu)\bn}=\bm{0},\qquad \jump{\bu}=-\bK\bsig(\bu)\bn \qquad\text{on }\Gamma,
\end{equation}
where the symmetric positive semi-definite tensor $\bK$ measures the compliancy of the
interface. For an isotropic interface, $\bK=\alpha\bI+(\beta-\alpha)\bn\otimes\bn$ with
$\alpha\ge0$ and $\beta\ge0$ the tangential and the normal compliancy. These conditions are
known as \emph{linear slip}, spring-type or imperfect interface conditions
\cite{Schoenberg1980,PhamMaurelMarigo2021}; they interpolate between the perfectly bonded
regime $\bK={\mathsf{0}}$ and the limit $\alpha,\beta\to+\infty$, in which the two subdomains
decouple and $\Gamma$ becomes a traction-free boundary for each of them. In the time domain
such interfaces act as frequency-dependent filters and are routinely used in seismic
modelling to represent fracture networks. Fitted discontinuous Galerkin methods have been applied to seismic modelling in the presence of fractures, without considering robustness with respect to the compliancy. We refer to 
\cite{de2016elastic,lambrecht2018nodal,MF19}.

The numerical approximation of the associated wave equation raises four distinct
difficulties. First, the exact solution is discontinuous across $\Gamma$, so that either the
mesh resolves the interface, or the discretisation must be able to represent the jump inside
the mesh cells. Second, the compliancy parameters typically vary over many orders of
magnitude, and a discretisation whose stability or accuracy degenerates in either of the two
extreme regimes is of little practical use. Third, if the interface is not resolved by the
mesh, the cells cut by $\Gamma$ may carry arbitrarily small subcells, which spoils the
conditioning of the discrete problem and, in the worst case, its stability. Fourth, the
scheme should preserve a discrete energy so that long-time simulations do not suffer from
spurious dissipation or growth.

These difficulties have been addressed separately in the literature. Hansbo and Hansbo
\cite{HansboHansbo2004} proposed an unfitted finite element method for the elasticity
problem \eqref{eq:slip}, based on a doubling of the degrees of freedom in the cells cut by
the interface and on a Nitsche-type formulation involving a regularised interface stiffness
$\bS_h=(h\delta^{-1}\bI+\bK)^{-1}$; their method treated the whole range of compliancies with
one and the same formulation and was shown to be optimally convergent, uniformly in $\alpha$
and $\beta$, for piecewise affine approximations; see also \cite{HansboLarson2022}. Hybrid
high-order (HHO) methods, introduced in \cite{DiPietroErnLemaire2014} for diffusion problems
and in \cite{DiPietroErn2015} for linear elasticity, provide arbitrary-order, locking-free
discretisations on general polytopal meshes; they are built from a local reconstruction
operator and a local stabilisation operator, they support hanging nodes and polytopal cells,
and the cell unknowns can be eliminated locally so that the global problem only couples the
face unknowns. Unfitted HHO discretisations of elliptic interface problems, in which the
interface is not resolved by the mesh and the small cut cells are handled by agglomeration,
were devised and analysed in \cite{BurmanErn2018,BCDE2021} and extended to the Stokes
interface problem in \cite{BurmanDelayErn2021}. The HHO discretisation of the wave equation
was studied numerically in \cite{BurmanDuranErn2022CAMC} and analysed in
\cite{BDES2021}, where optimal $H^1$- and $L^2$-error estimates were established for the
space semi-discrete second-order formulation in time; unfitted HHO methods for the scalar
wave equation were considered in \cite{BurmanDuranErn2022CMAME}. Hybrid methods including HDG and weak Galerkin were presented to discretize the space variables for the linear acoustic wave equation in \cite{ErnSteins2024,JanaMohapatraDeka2026}, second order time derivative was approximated by the leapfrog scheme and the implicit Newmark scheme, and optimal convergence in space and time was proved. Discontinuous Galerkin
methods for elastodynamics on polytopal meshes were developed in
\cite{AntoniettiMazzieri2018,AntoniettiBonaldiMazzieri2020}, and cut finite element methods
for the elastic wave equation were proposed in \cite{StickoLudvigssonKreiss2019}. Recently, a fitted HHO
method of arbitrary order covering the whole range of interface compliancies for the static
elasticity problem was proposed and analysed in \cite{HuangBurman2024}. A comprehensive study
of the elastodynamics problem discretised using the HHO method, including the
elasto-acoustic coupling, high-order Runge--Kutta time stepping and large-scale geophysical
applications, has recently been developed in \cite{Mottier2025,MottierErnGuillot2026}; unfitted HHO methods are
also considered therein, but only for scalar elliptic interface problems, the ill-cut cells
being cured by polynomial extension \cite{BurmanErnMottier2025} rather than by
agglomeration. 

The purpose of the present work is to combine these ingredients: we design and analyse an
unfitted HHO method of arbitrary order $k\ge1$ for the elastodynamics problem with a linear
slip interface. The method has two specific ingredients. The first one is a local symmetric
strain reconstruction in each cut subcell, in which the interface condition is built into the
definition of the operator on the softer side of the interface. 
 The second is an interface stabilisation weighted by
$\bS_h$, which penalises the jump of the {\emph{cell unknowns}} across $\Gamma$. 
The two extreme
regimes are then recovered automatically: when $\alpha=\beta=0$ one has $\bI-\bS_h\bK=\bI$
and $\bS_h=\delta h_T^{-1}\bI$, so that the interface stabilisation is a Nitsche penalty and
the reconstruction enforces the continuity of the displacement weakly, whereas for
$\alpha,\beta\gg h_T/\delta$ one has $\bS_h\simeq\bK^{-1}$, which is the physical interface
stiffness, and $\bI-\bS_h\bK\simeq\mathsf{0}$, so that the two subdomains decouple.

Our main results concern the space semi-discrete problem: the coercivity and the continuity
of the discrete bilinear form (\Cref{thm:stab}), an energy-error estimate of order $h^{k+1}$
(\Cref{prop:energy}) and an $L^2$-error estimate of order $h^{k+2}$ (\Cref{prop:L2}),
obtained through a Ritz-type projection in the spirit of \cite{Baker1976,Wheeler1973}, with
constants that do not depend on the compliancy parameters nor on how the interface cuts the
mesh. The auxiliary results on the wave equation that we use are borrowed from
\cite{BDES2021}, and those concerning unfitted HHO discretisations from \cite{BCDE2021}. The
space semi-discrete scheme is combined either
with the Newmark scheme, which conserves a discrete energy exactly, or with singly
diagonally implicit Runge--Kutta schemes of order up to four. The theoretical results are confirmed in \Cref{sec:num} by two-dimensional
numerical experiments on three test cases of increasing difficulty.

The rest of the paper is organised as follows. In \Cref{sec:model} we introduce the
elastodynamics problem with an imperfect interface and its weak formulation. The unfitted
HHO method is described in \Cref{sec:method}. \Cref{sec:prelim} collects the technical tools
used in the analysis. The stability and error analysis is carried out in \Cref{sec:analysis}
and the time discretisation is discussed in \Cref{sec:time}. Numerical experiments are
presented in \Cref{sec:num}, and conclusions are drawn in \Cref{sec:concl}.

\section{An elastodynamics problem with an imperfect interface}\label{sec:model}

Let $\Omega$ be a connected, convex polygonal/polyhedral domain in $\mathbb{R}^d$,
$d\in\{2,3\}$, with boundary $\partial\Omega$, and let $\Gamma$ be a smooth internal
interface of class $C^2$ dividing $\Omega$ into two subdomains $\Omega_1$ and $\Omega_2$, so
that $\overline{\Omega}=\overline{\Omega_1}\cup\overline{\Omega_2}$ and
$\Gamma=\partial\Omega_1\cap\partial\Omega_2$; see \Cref{fig:domain}. Throughout the paper we
use subscripts to denote the restriction of a function to a subdomain, and vectors and
tensors are typed in bold face. For any sufficiently regular function $\bu$ in
$\Omega_1\cup\Omega_2$ we define the jump of $\bu$ across $\Gamma$ by
$\jump{\bu}=\bu_1|_\Gamma-\bu_2|_\Gamma$, and $\bn=\bn_\Gamma$ denotes the unit normal
vector to $\Gamma$ pointing from $\Omega_1$ to $\Omega_2$.

For an open bounded domain $D\subset\mathbb{R}^d$ we use the standard notation for the
Sobolev spaces $H^m(D)$, $m\ge0$, equipped with the norm $\|\cdot\|_{H^m(D)}$ and seminorm
$|\cdot|_{H^m(D)}$, and we denote the corresponding spaces of vector-valued functions and of
symmetric tensor-valued functions by $\bm{H}^m(D)=[H^m(D)]^d$ and
$\bm{\mathcal{H}}^m(D)=[H^m(D)]^{d\times d}_{\rm sym}$, respectively. We set
$\bm{H}^1(\Omega_1\cup\Omega_2)=\{\bv\in\bm{L}^2(\Omega)\mid
\bv|_{\Omega_i}\in\bm{H}^1(\Omega_i),\ i=1,2\}$ and
$\bm{H}^1_0(\Omega_1\cup\Omega_2)=\{\bv\in\bm{H}^1(\Omega_1\cup\Omega_2)\mid
\bv|_{\partial\Omega}=\bm{0}\}$. Let $J=(0,T)$ be the time interval with $T>0$. Given
$\ell\in\mathbb{N}$ and a Hilbert space $\mathbb{H}$, we write $C^\ell(\overline{J};\mathbb{H})$
for the space of $\mathbb{H}$-valued functions that are $\ell$ times continuously
differentiable in $J$.

\begin{figure}[htbp]\centering
\includegraphics[width=\linewidth]{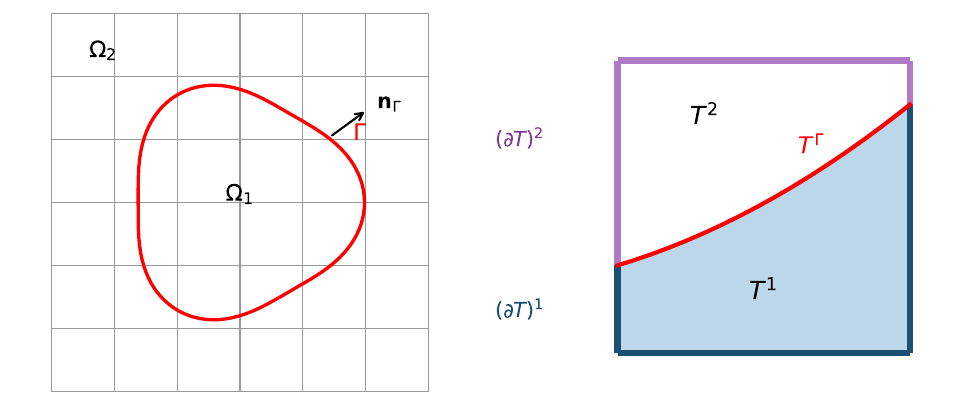}
\caption{Left: a domain $\Omega$ divided into the subdomains $\Omega_1$ and $\Omega_2$ by the
interface $\Gamma$, which is not resolved by the background mesh. Right: notation for a cut
cell $T\in\Th^\Gamma$; $(\partial T)^1$ is drawn in dark blue and $(\partial T)^2$ in
purple.}\label{fig:domain}
\end{figure}

We consider the following elastodynamics problem with discontinuities in the Lam\'e
parameters along the material interface $\Gamma$: find the displacement
$\bu:J\times\Omega\to\mathbb{R}^d$ such that
\begin{equation}\label{eq:model}
\left\{
\begin{aligned}
\bsig(\bu)&=\CC\beps(\bu)=2\mu\beps(\bu)+\lambda(\nabla\!\cdot\!\bu)\bI
&&\text{in }J\times(\Omega_1\cup\Omega_2),\\
\varrho\,\dtt\bu-\nabla\!\cdot\!\bsig(\bu)&=\bff
&&\text{in }J\times(\Omega_1\cup\Omega_2),\\
\jump{\bu}+\bK\bsig(\bu_1)\bn&=\bg_D,\qquad
\jump{\bsig(\bu)\bn}=\bg_N &&\text{on }J\times\Gamma,\\
\bu(0)=\bu_0,\quad \dt\bu(0)&=\bq_0 &&\text{in }\Omega_1\cup\Omega_2,\\
\bu&=\bm{0}&&\text{on }J\times\partial\Omega,
\end{aligned}
\right.
\end{equation}
where $\CC=2\mu\bm{\mathbb{I}}+\lambda\bI tr$ is the fourth-order elasticity tensor,
$\varrho>0$ is the mass density, $\bff\in L^2(J;\bm{L}^2(\Omega))$ is the external force,
$\bg_D\in L^2(J;\bm{H}^{1/2}(\Gamma))$, $\bg_N\in L^2(J;\bm{L}^2(\Gamma))$,
$\bu_0\in\bm{H}^1_0(\Omega_1\cup\Omega_2)$ and
$\bq_0\in\bm{H}^1_0(\Omega_1\cup\Omega_2)$. The strain tensor is
$\beps(\bu)=\frac12(\nabla\bu+\nabla\bu^{\mathsf T})$ and $\bI=[\delta^{ij}]_{i,j=1}^d$ is
the identity tensor. The Lam\'e parameters are
\[
\lambda=\frac{E\nu}{(1-2\nu)(1+\nu)},\qquad \mu=\frac{E}{2(1+\nu)},
\]
with $E$ the Young modulus and $\nu$ the Poisson ratio; they satisfy $0<c\le\mu\le C$ and
$0\le\lambda\le C$, so that the incompressible case is excluded. The coefficients
$\varrho,\mu,\lambda$ are piecewise constant on the partition $\{\Omega_1,\Omega_2\}$, and
we write $\mu_i,\lambda_i,\varrho_i$ for their values in $\Omega_i$ and
$\kappa_i=\lambda_i/\mu_i$. Without loss of generality we assume that
\begin{equation}\label{eq:softside}
\mu_1\le\mu_2,
\end{equation}
i.e.\ $\Omega_1$ is the softer material; the traction appearing in the first interface
condition of \eqref{eq:model} is then evaluated on the softer side. We finally set
\begin{equation}\label{eq:delta}
\delta:=2\mu_1+d\lambda_1 .
\end{equation}

\begin{remark}[Nonhomogeneous interface condition]\label{rem:nonhomoge IC}
The first interface condition of \eqref{eq:model} can be written in the more general form
$\jump{\bu}+\bK\{\bsig(\bu)\bn\}_\omega=\tilde\bg_D$ with the weighted average
$\{\bv\}_\omega=\omega_1\bv_1+\omega_2\bv_2$, $\omega_1+\omega_2=1$. Since the traction jump
is prescribed, $\{\bsig(\bu)\bn\}_\omega=\bsig(\bu_1)\bn-\omega_2\bg_N$, so that all these
formulations coincide up to a modification of the right data with $\bg_D=\tilde\bg_D+\omega_2K\bg_N$, and we work with
$\omega=(1,0)$ throughout. The choice \eqref{eq:softside} of the softer side is the natural
one in the highly contrasted case $\mu_1\ll\mu_2$, in the spirit of
\cite{BCDE2021,ErnStephansenZunino2009}.
\end{remark}

A reasonable functional setting for the weak solution of \eqref{eq:model} is
$\bu\in L^2(J;\bm{H}^1_0(\Omega_1\cup\Omega_2))$, $\dt\bu\in L^2(J;\bm{L}^2(\Omega))$ and
$\dtt\bu\in L^2(J;\bm{H}^{-1}(\Omega_1\cup\Omega_2))$. Assuming
$\bu\in H^2(J;\bm{L}^2(\Omega))$, we have for a.e.\ $t\in J$
\begin{equation}\label{eq:weak}
(\varrho\,\dtt\bu(t),\bw)_\Omega+a(\bu(t),\bw)=\ell(\bw)
\qquad\forall\,\bw\in\bm{H}^1_0(\Omega_1\cup\Omega_2),
\end{equation}
where
\begin{align}
a(\bv,\bw)&=\sum_{i=1,2}\Big(2\mu_i(\beps(\bv),\beps(\bw))_{\Omega_i}
+\lambda_i(\nabla\!\cdot\!\bv,\nabla\!\cdot\!\bw)_{\Omega_i}\Big)
+\big\langle \bK^{-1}\jump{\bv},\jump{\bw}\big\rangle_\Gamma,\label{eq:a}\\
\ell(\bw)&=(\bff,\bw)_\Omega+\big\langle \bK^{-1}\bg_D,\jump{\bw}\big\rangle_\Gamma
+\langle\bg_N,\bw_2\rangle_\Gamma.\label{eq:ell}
\end{align}
Here $(\cdot,\cdot)_D$ and $\langle\cdot,\cdot\rangle_S$ denote the $L^2$-inner products on a
domain $D\subset\mathbb{R}^d$ and on a manifold $S$ of codimension one. As in
\cite{HansboHansbo2004,HuangBurman2024}, when $\bK$ is singular the two terms involving
$\bK^{-1}$ are understood in the sense of the pseudo-inverse of $\bK$ and the trial and test
space is restricted to the fields whose jump lies in the range of $\bK$; the perfectly bonded
case $\bK=\mathsf{0}$ then corresponds to $\jump{\bv}=\bg_D$. The discretisation designed below
covers all these situations with one and the same formulation, without any case distinction.

\section{The unfitted HHO method}\label{sec:method}

\subsection{Unfitted meshes}\label{ssec:meshes}

Let $\{\Th\}_{h>0}$ be a sequence of shape-regular meshes of $\Omega$ with mesh size
$h=\max_{T\in\Th}h_T$, where $h_T$ is the diameter of the cell $T$. The mesh $\Th$ is
composed of nonempty disjoint cells such that $\overline\Omega=\bigcup_{T\in\Th}\overline T$.
The mesh cells are conventionally open subsets of $\mathbb{R}^d$ (not necessarily convex) and
they can have a polygonal/polyhedral shape with straight edges (if $d=2$) or planar faces (if
$d=3$); this setting in particular allows for meshes with hanging nodes, and hence for meshes
produced by a cell agglomeration procedure. The mesh faces are collected in the set $\Fh$. In
a nutshell, each mesh $\Th$ admits a matching simplicial submesh $\Th'$ such that any cell
(or face) of $\Th'$ is a subset of a cell (or face) of $\Th$, and there is a mesh-regularity
parameter $\rho>0$ such that, for all $h>0$, any $T\in\Th$ and any $S\in\Th'$ with $S\subset T$,
one has $\rho h_S\le r_S$ and $\rho h_T\le h_S$, where $r_S$ denotes the inradius of $S$.

The meshes are \emph{not} fitted to $\Gamma$. We define the partition
$\Th=\Th^1\cup\Th^\Gamma\cup\Th^2$, where
\begin{equation}\label{eq:cellpart}
\Th^i:=\{T\in\Th\mid T\subset\Omega_i\},\quad i=1,2,\qquad
\Th^\Gamma:=\{T\in\Th\mid T\cap\Gamma\neq\emptyset\}
\end{equation}
collect, respectively, the uncut cells lying inside $\Omega_i$ and the cells cut by the
interface. For each cut cell $T\in\Th^\Gamma$ and all $i=1,2$ we set
\begin{equation}\label{eq:subcells}
T^i:=T\cap\Omega_i,\qquad T^\Gamma:=T\cap\Gamma,
\end{equation}
and the boundary of the subcell $T^i$ is decomposed as
$\partial(T^i)=(\partial T)^i\cup T^\Gamma$ with $(\partial T)^i:=\partial T\cap\Omega_i$. In
order to unify the notation, for every uncut cell $T\in\Th^i$, $i=1,2$, we set
\[
T^i:=T,\quad T^{\bar i}:=\emptyset,\quad (\partial T)^i:=\partial T,\quad
(\partial T)^{\bar i}:=\emptyset,\quad T^\Gamma:=\emptyset,
\]
where $\bar i=3-i$. Similarly, for $F\in\Fh$ we set $F^i:=F\cap\Omega_i$ and
$\mathcal{F}_{(\partial T)^i}:=\{F^i\mid F\in\mathcal{F}_{\partial T}\}$, where
$\mathcal{F}_{\partial T}$ is the collection of the faces composing $\partial T$.

Three additional assumptions on the meshes are needed, as is classical for unfitted methods
\cite{BurmanErn2018,BCDE2021}. The first one quantifies how irregular the mesh cells can
become due to the intersecting interface (and provides the discrete inverse inequalities of
\Cref{lem:disc}), the second one quantifies how well the mesh resolves the interface (and
provides the multiplicative trace inequality of \Cref{lem:mtr}), and the third one requires
the meshes to be not too graded. Let $B(\bx,r)$ be the ball of centre $\bx\in\mathbb{R}^d$
and radius $r>0$.

\begin{assumption}[Cut cells]\label{ass:cut}
There is $\zeta\in(0,1)$ such that, for any $T\in\Th^\Gamma$ and all $i=1,2$, there is
$\check\bx_{T^i}\in T^i$ such that $B(\check\bx_{T^i},\zeta h_T)\subset T^i$.
\end{assumption}

\begin{assumption}[Resolving $\Gamma$]\label{ass:resolve}
There is $\gamma\in(0,1)$ such that, for any $T\in\Th^\Gamma$, there is a point
$\check\bx_T\in\mathbb{R}^d$ such that, setting $T^\dagger:=B(\check\bx_T,\gamma^{-1}h_T)$,
the following holds: (i) $T\subset T^\dagger$; (ii) for any $\bm{s}\in T^\Gamma$,
$d(\check\bx_T,T_{\bm{s}}\Gamma)\ge\gamma h_T$, where $T_{\bm{s}}\Gamma$ is the tangent plane
to $\Gamma$ at $\bm{s}$; (iii) for any $F\in\mathcal{F}_{\partial T}$, there is
$\bx_F\in T^\dagger$ such that $d(\bx_F,F)\ge\gamma h_T$.
\end{assumption}

\begin{assumption}[Mild mesh grading]\label{ass:grading}
For any $T\in\Th$, let the neighbouring layers $\triangle_j(T)\subset\mathbb{R}^d$ be defined
by induction as $\triangle_0(T):=T$ and
$\triangle_{j+1}(T):=\bigcup\{\overline{T'}\mid T'\in\Th,\ \overline{T'}\cap
\overline{\triangle_j(T)}\ne\emptyset\}$. There is $n_0\in\mathbb{N}$ such that, for all
$T\in\Th$, the ball $T^\dagger$ of \Cref{ass:resolve} satisfies
$T^\dagger\subset\triangle_{n_0}(T)$.
\end{assumption}

\Cref{ass:resolve,ass:grading} are satisfied as soon as the mesh is fine enough with respect
to the curvature of $\Gamma$, see \cite[Lem.~6.1 and Lem.~3.3]{BurmanErn2018}.
\Cref{ass:cut} is \emph{not} satisfied by a generic background mesh, since the interface may
cut off an arbitrarily small portion of a cell; it is enforced by agglomerating the cells
carrying a small cut to a neighbouring cell, as described in \cite[Sect.~4.3]{BCDE2021},
which produces a (locally) polyhedral mesh in the sense of \Cref{ssec:meshes} and which
guarantees that the agglomeration does not propagate beyond nearest neighbours. The
agglomeration procedure is purely geometric and independent of the discretisation; we use a
simplified two-stage version of it in \Cref{sec:num}. An alternative, less intrusive cure
for the ill-cut cells, which leaves the mesh unchanged and replaces the agglomeration by a
polynomial extension in the definition of the reconstruction operator, has been proposed in
\cite{BurmanErnMottier2025,Mottier2025}; it would be of interest to combine it with the
present formulation.

\subsection{Discrete unknowns and the regularised interface stiffness}\label{ssec:unknowns}

Let $k\ge1$ be the polynomial degree. The discrete unknowns for the displacement are
piecewise polynomials of degree $k$ attached to the mesh faces and of degree $k+1$ attached
to the mesh cells; the degree $k\ge1$ is needed for the local Korn inequality of
\Cref{lem:disc}. For any subset $S\subset\mathbb{R}^d$ consisting of one mesh (sub)cell or
one mesh (sub)face, and for any $l\in\mathbb{N}$, we denote by $\Pk{l}(S)$ (resp.\
$\bPk{l}(S)$, $\sPk{l}(S)$) the space of scalar-valued (resp.\ vector-valued,
symmetric-matrix-valued) polynomials in $S$ of degree at most $l$. 
Whenever $S=\emptyset$ we abuse the notation by writing $\Pk{l}(S)=\{0\}$ and
$(\cdot,\cdot)_S=0$.

Let $T\in\Th$. For all $i=1,2$ we set
$\Pk{k}(\mathcal{F}_{(\partial T)^i}):=\times_{F\in\mathcal{F}_{\partial T}}\Pk{k}(F^i)$ and
we define the local discrete unknowns as
\begin{equation}\label{eq:localdofs}
\bhv_T:=\big(\bv_{T^1},\bv_{T^2},\bv_{(\partial T)^1},\bv_{(\partial T)^2}\big)\in\VT
:=\bPk{k+1}(T^1)\times\bPk{k+1}(T^2)\times
\bPk{k}(\mathcal{F}_{(\partial T)^1})\times\bPk{k}(\mathcal{F}_{(\partial T)^2}).
\end{equation}
It is also convenient to regroup the local unknowns on each side of the interface by means of
the notation $\bhv_{T^i}:=(\bv_{T^i},\bv_{(\partial T)^i})\in
\hat{\bm{V}}_{T^i}:=\bPk{k+1}(T^i)\times\bPk{k}(\mathcal{F}_{(\partial T)^i})$, $i=1,2$. Note
that \emph{no discrete unknown is attached to $T^\Gamma$}: for any $T\in\Th^\Gamma$ and any
$\bhv_T\in\VT$, the discrete jump
\begin{equation}\label{eq:djump}
\jump{\bv_T}_\Gamma:=\bv_{T^1}|_{T^\Gamma}-\bv_{T^2}|_{T^\Gamma}
\end{equation}
is recovered \emph{a posteriori} from the two cell unknowns.

In order to design a method that is robust with respect to the discontinuous Lam\'e
parameters and with respect to all the compliancy parameters of $\bK$, we define, following
\cite{HansboHansbo2004}, the \emph{regularised interface stiffness}
\begin{equation}\label{eq:Sh}
\bS_h|_T:=\Big(\frac{h_T}{\delta}\bI+\bK\Big)^{-1},\qquad \forall\,T\in\Th,
\end{equation}
with $\delta=2\mu_1+d\lambda_1$ as in \eqref{eq:delta}. Since $\bK$ is symmetric positive
semi-definite and $h_T/\delta>0$, the matrix $\bS_h|_T$ is well defined, symmetric positive
definite, and commutes with $\bK$, for every value of the compliancy parameters; no case
distinction and no perturbation of $\bK$ is needed.

\subsection{Local reconstruction and stabilisation}\label{ssec:local}

The first key ingredient of the HHO method is a local symmetric strain reconstruction in each
mesh cell. For any $T\in\Th$ and $i=1,2$, we define the local discrete symmetric gradient
reconstruction operator $\bE_{T^i}:\VT\to\sPk{k}(T^i)$ as follows. On the softer side, the
interface condition is built into the operator: $\bE_{T^1}(\bhv_T)$ is such that, for all
$\btau\in\sPk{k}(T^1)$,
\begin{multline}\label{eq:E1}
(\bE_{T^1}(\bhv_T),\btau)_{T^1}
+\big\langle(\bI-\bS_h\bK)\bK\,\bG_{T^1}(\bhv_T)\bn,\btau\bn\big\rangle_{T^\Gamma}\\
=(\beps(\bv_{T^1}),\btau)_{T^1}
+\big\langle\bv_{(\partial T)^1}-\bv_{T^1},\btau\bn_T\big\rangle_{(\partial T)^1}
-\big\langle(\bI-\bS_h\bK)\jump{\bv_T}_\Gamma,\btau\bn\big\rangle_{T^\Gamma},
\end{multline}
where $\bn_T$ is the unit outward normal to $T$ and where we have set
\begin{equation}\label{eq:G}
\bG_{T^i}(\bhv_T):=\CC_i\bE_{T^i}(\bhv_T)
=2\mu_i\bE_{T^i}(\bhv_T)+\lambda_i\,\tr\bE_{T^i}(\bhv_T)\,\bI .
\end{equation}
On the stiffer side, no interface contribution is present: $\bE_{T^2}(\bhv_T)$ is such that,
for all $\btau\in\sPk{k}(T^2)$,
\begin{equation}\label{eq:E2}
(\bE_{T^2}(\bhv_T),\btau)_{T^2}=(\beps(\bv_{T^2}),\btau)_{T^2}
+\big\langle\bv_{(\partial T)^2}-\bv_{T^2},\btau\bn_T\big\rangle_{(\partial T)^2}.
\end{equation}
Note that $\bE_{T^2}$ only depends on $\bhv_{T^2}$, whereas $\bE_{T^1}$ depends on both
$\bhv_{T^1}$ and $\bv_{T^2}$, through the discrete jump \eqref{eq:djump}. Taking
$\btau=q\bI$ in \eqref{eq:E1}--\eqref{eq:E2} shows that the local discrete divergence
reconstruction $D_{T^i}(\bhv_T):=\tr\bE_{T^i}(\bhv_T)\in\Pk{k}(T^i)$ satisfies, for all
$q\in\Pk{k}(T^1)$,
\begin{multline}\label{eq:D1}
(D_{T^1}(\bhv_T),q)_{T^1}
+\big\langle(\bI-\bS_h\bK)\bK\,\bG_{T^1}(\bhv_T)\bn,q\bn\big\rangle_{T^\Gamma}\\
=(\nabla\!\cdot\!\bv_{T^1},q)_{T^1}
+\big\langle\bv_{(\partial T)^1}-\bv_{T^1},q\bn_T\big\rangle_{(\partial T)^1}
-\big\langle(\bI-\bS_h\bK)\jump{\bv_T}_\Gamma,q\bn\big\rangle_{T^\Gamma},
\end{multline}
and similarly for $D_{T^2}$.

\begin{lemma}[Well-posedness of the reconstruction]\label{lem:wp}
For any $T\in\Th$ and any $\bhv_T\in\VT$, the linear systems \eqref{eq:E1} and \eqref{eq:E2}
have a unique solution.
\end{lemma}

\begin{proof}
Both are square linear systems, so it suffices to prove injectivity. For \eqref{eq:E2} this
is clear. For \eqref{eq:E1}, let $\bE:=\bE_{T^1}(\bhv_T)$ solve \eqref{eq:E1} with zero
right-hand side and take $\btau=\CC_1\bE=\bG_{T^1}(\bhv_T)$, which is an admissible test
function. This gives
\[
(\bE,\CC_1\bE)_{T^1}
+\big\langle(\bI-\bS_h\bK)\bK\,\CC_1\bE\bn,\CC_1\bE\bn\big\rangle_{T^\Gamma}=0 .
\]
Since $\CC_1$ is symmetric positive definite and, by \eqref{eq:ImSK} below,
$(\bI-\bS_h\bK)\bK=\tfrac{h_T}{\delta}\bS_h\bK$ is symmetric positive semi-definite, both
terms are nonnegative, whence $\bE=\bm{0}$.
\end{proof}

The second key ingredient is the local stabilisation operator
${\bf S}:\hat{\bm{V}}_{T^i}\to\bPk{k}(\mathcal{F}_{(\partial T)^i})$, used to penalise in a
least-squares sense the difference between the face unknown $\bv_{(\partial T)^i}$ and the
trace of the cell unknown $\bv_{T^i}$. Let $\bm{\Pi}_{(\partial T)^i}$ be the
$L^2$-orthogonal projection onto $\bPk{k}(\mathcal{F}_{(\partial T)^i})$. For any
$\bhv_T,\bhw_T\in\VT$ we set
\begin{equation}\label{eq:S}
{\bf S}(\bhv_{T^i}):=\bv_{(\partial T)^i}-\bm{\Pi}_{(\partial T)^i}(\bv_{T^i}),\qquad i=1,2,
\end{equation}
and define the two stabilisation bilinear forms
\begin{equation}\label{eq:stab}
\begin{aligned}
s_T^{\partial T}(\bhv_T,\bhw_T)&:=\sum_{i=1,2}\frac{2\mu_i+d\lambda_i}{h_T}
\big\langle{\bf S}(\bhv_{T^i}),{\bf S}(\bhw_{T^i})\big\rangle_{(\partial T)^i},\\
s_T^\Gamma(\bv_T,\bw_T)&:=\big\langle\bS_h\jump{\bv_T}_\Gamma,\jump{\bw_T}_\Gamma
\big\rangle_{T^\Gamma},
\end{aligned}
\end{equation}
and we set $s_T:=s_T^{\partial T}+s_T^\Gamma$. Finally, we mimic locally the exact bilinear
form $a(\cdot,\cdot)$ of \eqref{eq:a} by means of the local bilinear form on $\VT\times\VT$
\begin{multline}\label{eq:aT}
a_T(\bhv_T,\bhw_T):=\sum_{i=1,2}\Big(2\mu_i\big(\bE_{T^i}(\bhv_T),\bE_{T^i}(\bhw_T)\big)_{T^i}
+\lambda_i\big(D_{T^i}(\bhv_T),D_{T^i}(\bhw_T)\big)_{T^i}\Big)\\
+\big\langle(\bI-\bS_h\bK)\bK\,\bG_{T^1}(\bhv_T)\bn,\bG_{T^1}(\bhw_T)\bn\big\rangle_{T^\Gamma}.
\end{multline}
The bilinear form $a_T$ is symmetric and positive semi-definite, and testing \eqref{eq:E1}
and \eqref{eq:E2} with $\btau=\bG_{T^i}(\bhw_T)$ gives the computable representation
\begin{multline}\label{eq:aTcomp}
a_T(\bhv_T,\bhw_T)=\sum_{i=1,2}\Big[(\beps(\bv_{T^i}),\bG_{T^i}(\bhw_T))_{T^i}
+\big\langle\bv_{(\partial T)^i}-\bv_{T^i},\bG_{T^i}(\bhw_T)\bn_T
\big\rangle_{(\partial T)^i}\Big]\\
-\big\langle(\bI-\bS_h\bK)\jump{\bv_T}_\Gamma,\bG_{T^1}(\bhw_T)\bn\big\rangle_{T^\Gamma}.
\end{multline}

\begin{remark}[The two extreme regimes]\label{rem:regimes}
The definition \eqref{eq:Sh} is designed so that one single formulation covers the whole
range of compliancies. When $\alpha=\beta=0$, i.e.\ for a perfectly bonded interface, one has
$\bK= \mathsf{0}$, hence $\bI-\bS_h\bK=\bI$ and $\bS_h=\delta h_T^{-1}\bI$: the interface terms in
\eqref{eq:E1} then enforce weakly the continuity of the displacement, and $s_T^\Gamma$ is a
Nitsche-type penalty of size $\delta/h_T$. In the opposite regime
$\alpha,\beta\gg h_T/\delta$, one has $\bS_h\simeq\bK^{-1}$, so that $s_T^\Gamma$ represents
the physical interface stiffness, and $\bI-\bS_h\bK\simeq \mathsf{0}$, so that the interface
contributions to \eqref{eq:E1} and \eqref{eq:aT} disappear; in the limit
$\alpha,\beta\to+\infty$ the discrete problem decouples into two independent problems with a
traction-free boundary condition on $\Gamma$, which is the correct behaviour of
\eqref{eq:model} in that limit. Both limits, and the degenerate cases $\alpha=0$ or
$\beta=0$, are uniformly accessible; this is confirmed numerically in \Cref{ssec:robust}.
\end{remark}

\subsection{The space semi-discrete scheme}\label{ssec:semidiscrete}

For all $i=1,2$ we set
$\bm{V}_{\mathcal{T}^i}=\times_{T\in\Th}\bPk{k+1}(T^i)$ and
$\bm{V}_{\mathcal{F}^i}=\times_{F\in\Fh}\bPk{k}(F^i)$, and we define the global discrete
spaces
\[
\bm{V}_\mathcal{T}=\bm{V}_{\mathcal{T}^1}\times\bm{V}_{\mathcal{T}^2},\qquad
\bm{V}_\mathcal{F}=\bm{V}_{\mathcal{F}^1}\times\bm{V}_{\mathcal{F}^2},\qquad
\Vh=\bm{V}_\mathcal{T}\times\bm{V}_\mathcal{F},
\]
and we denote by $\Vhz$ the subspace of $\Vh$ where all the degrees of freedom attached to
the faces composing $\partial\Omega$ are null. Notice that the role of
$s_T^{\partial T}(\cdot,\cdot)$ is to enforce weakly the matching between the cell-based and
the face-based HHO unknowns on all the faces inside the subdomains, whereas the role of
combining $s_T^\Gamma(\cdot,\cdot)$ and the interface term in $a_T(\cdot,\cdot)$ is to
enforce weakly the first interface condition.

The space semi-discrete HHO scheme for the second-order formulation in time of the elastic
wave equation consists in finding
$\bhu_h=(\bu_\mathcal{T},\bu_\mathcal{F})\in C^2(\overline J;\Vhz)$ such that, for all
$t\in\overline J$,
\begin{equation}\label{eq:semi}
\big(\varrho\,\dtt\bu_\mathcal{T}(t),\bw_\mathcal{T}\big)_\Omega
+a_h\big(\bhu_h(t),\bhw_h\big)=\ell_h(\bhw_h)
\qquad\forall\,\bhw_h=(\bw_\mathcal{T},\bw_\mathcal{F})\in\Vhz,
\end{equation}
where
\begin{equation}\label{eq:ah}
a_h(\bhv_h,\bhw_h):=\sum_{T\in\Th}\big[a_T(\bhv_T,\bhw_T)+s_T(\bhv_T,\bhw_T)\big],
\qquad \ell_h(\bhw_h):=\sum_{T\in\Th}\ell_T(\bhw_T),
\end{equation}
and
\begin{multline}\label{eq:lT}
\ell_T(\bhw_T):=\sum_{i=1,2}(\bff(t),\bw_{T^i})_{T^i}
-\big\langle(\bI-\bS_h\bK)\bg_D,\bG_{T^1}(\bhw_T)\bn\big\rangle_{T^\Gamma}\\
+\big\langle\bS_h\bg_D,\jump{\bw_T}_\Gamma\big\rangle_{T^\Gamma}
+\langle\bg_N,\bw_{T^2}\rangle_{T^\Gamma}.
\end{multline}
The initial conditions for \eqref{eq:semi} only concern the cell component
$\bu_\mathcal{T}$ and are
\begin{equation}\label{eq:ic}
\bu_\mathcal{T}(0)|_{T^i}=\bm{I}_{T^i}(\bu_0|_{\Omega_i}),\qquad
\dt\bu_\mathcal{T}(0)|_{T^i}=\bm{I}_{T^i}(\bq_0|_{\Omega_i}),\qquad i=1,2,\ T\in\Th,
\end{equation}
where $\bm{I}_{T^i}$ is the cell interpolation operator defined in \eqref{eq:interp}, which
coincides with the $L^2$-orthogonal projection onto $\bPk{k+1}(T)$ on every uncut cell.
Using instead the $L^2$-orthogonal projection onto $\bPk{k+1}(T^i)$ on the cut cells, as we
do in the computations of \Cref{sec:num}, only adds a term of order $h^{k+2}$ to the
estimates below. The boundary condition is encoded in the fact that
$\bhu_h(t)\in\Vhz$ for any $t\in\overline J$. Notice that, since the space semi-discrete
solution is smooth in time, \eqref{eq:semi} holds at the initial time, which implies that
$\bu_\mathcal{F}(0)\in\bm{V}_\mathcal{F}$ is uniquely determined by the equations
$a_h\big((\bu_\mathcal{T}(0),\bu_\mathcal{F}(0)),(\bm{0},\bw_\mathcal{F})\big)=
\ell_h((\bm{0},\bw_\mathcal{F}))$ for all $\bw_\mathcal{F}\in\bm{V}_\mathcal{F}$ with
$\bu_\mathcal{T}(0)$ specified in \eqref{eq:ic} and with the boundary condition enforcing to
zero the components of $\bu_\mathcal{F}(0)$ attached to the boundary faces; the same holds
for $\dt\bu_\mathcal{F}(0)$ and $\dtt\bu_\mathcal{F}(0)$.

\section{Preliminaries}\label{sec:prelim}

In what follows, we use the convention $A\lesssim B$ to abbreviate the inequality
$A\le CB$ for positive real numbers $A$ and $B$, where the constant $C$ only depends on the
polynomial degree $k\ge1$, on the mesh parameters $\rho,\zeta,\gamma,n_0$ and on the
constants $c_{\rm disc}$ and $c_{\rm mtr}$ introduced below, but neither on the Lam\'e
coefficients $\lambda_i,\mu_i$, $i=1,2$, nor on the compliancy tensor $\bK$, nor on the mesh
size $h>0$, nor on how the interface cuts the mesh.

\subsection{The regularised interface stiffness}\label{ssec:Shprop}

We first collect the properties of the matrix $\bS_h$ defined in \eqref{eq:Sh}; they are the
counterpart of \cite[Lem.~2]{HansboHansbo2004}.

\begin{lemma}[Properties of $\bS_h$]\label{lem:Sh}
Let $\bS_h=(h_T\delta^{-1}\bI+\bK)^{-1}$ and let $|\cdot|$ denote the matrix norm induced by
the Euclidean norm on $\mathbb{R}^d$. Then $\bS_h$, $\bK$ and $\bS_h\bK$ are symmetric and
commute, $\bS_h\bK$ is positive semi-definite, and
\begin{align}
\bI-\bS_h\bK&=\tfrac{h_T}{\delta}\bS_h,\label{eq:ImSK}\\
|\bS_h|&\le\tfrac{\delta}{h_T},\label{eq:Shbound}\\
|\bI-\bS_h\bK|\le1,&\qquad |\bS_h\bK|\le1,\label{eq:unitbound}\\
|(\bI-\bS_h\bK)\bK|&\le\tfrac{h_T}{\delta}.\label{eq:ImSKK}
\end{align}
\end{lemma}

\begin{proof}
Since $\bK$ is symmetric positive semi-definite, it is diagonalisable in an orthonormal basis
with nonnegative eigenvalues $\{k_j\}_{j=1}^d$ (for the isotropic compliancy considered here,
$k_1=\beta$ in the normal direction and $k_j=\alpha$, $j\ge2$, in the tangential ones). In
that basis, $\bS_h$ is diagonal with entries $(h_T/\delta+k_j)^{-1}$, whence the symmetry,
the commutation and \eqref{eq:Shbound}. Multiplying $\bS_h(h_T\delta^{-1}\bI+\bK)=\bI$ by
$\bS_h^{-1}$ on the left gives \eqref{eq:ImSK}. The eigenvalues of $\bS_h\bK$ are
$k_j/(h_T/\delta+k_j)\in[0,1)$, which gives \eqref{eq:unitbound}, and those of
$(\bI-\bS_h\bK)\bK=\tfrac{h_T}{\delta}\bS_h\bK$ are
$\tfrac{h_T}{\delta}k_j/(h_T/\delta+k_j)\in[0,h_T/\delta)$, which gives \eqref{eq:ImSKK}.
\end{proof}

\subsection{Discrete and trace inequalities}\label{ssec:ineq}

The role of \Cref{ass:cut} is to provide the following discrete inequalities on the cut
subcells; see \cite[Lem.~3.4]{BCDE2021} and \cite[Lem.~3.4]{BurmanErn2018}. Let
$RM:=\{\bm{r}\in\bPk{1}(\mathbb{R}^d)\mid\beps(\bm{r})=\bm{0}\}$ denote the space of
rigid-body motions.

\begin{lemma}[Discrete inequalities]\label{lem:disc}
Let \Cref{ass:cut} be fulfilled and let $l\in\mathbb{N}$. There is $c_{\rm disc}>0$,
depending on $\rho,\zeta$ and $l$, such that for any $T\in\Th$, all $i=1,2$ and any
$v_{T^i}\in\Pk{l}(T^i)$, the following inequalities hold true:
\begin{itemize}[leftmargin=1.2em,itemsep=1pt]
\item \emph{(discrete trace inequality)}
$\|v_{T^i}\|_{(\partial T)^i\cup T^\Gamma}\le c_{\rm disc}h_T^{-1/2}\|v_{T^i}\|_{T^i}$;
\item \emph{(discrete inverse inequality)}
$\|\nabla v_{T^i}\|_{T^i}\le c_{\rm disc}h_T^{-1}\|v_{T^i}\|_{T^i}$;
\item \emph{(discrete Poincar\'e inequality)} if $(v_{T^i},1)_{B(\check\bx_{T^i},h_T)}=0$,
then $\|v_{T^i}\|_{T^i}\le c_{\rm disc}h_T\|\nabla v_{T^i}\|_{T^i}$;
\item \emph{(discrete Korn inequality)} for $\bv_{T^i}\in\bPk{l}(T^i)$,
$\|\nabla\bv_{T^i}\|_{T^i}\le c_{\rm disc}\|\beps(\bv_{T^i})\|_{T^i}$ whenever
$(\bv_{T^i},\bm{r})_{B(\check\bx_{T^i},h_T)}=0$ for all $\bm{r}\in RM$.
\end{itemize}
\end{lemma}

The role of \Cref{ass:resolve} is to provide a multiplicative trace inequality; see
\cite[Lem.~3.3]{BurmanErn2018} for the proof.

\begin{lemma}[Multiplicative trace inequality]\label{lem:mtr}
Let \Cref{ass:resolve} be fulfilled. There is $c_{\rm mtr}>0$, depending on $\rho$ and
$\gamma$, such that for any $T\in\Th$, any $v\in H^1(T^\dagger)$ and all $i=1,2$,
\[
\|v\|_{(\partial T)^i\cup T^\Gamma}\le c_{\rm mtr}\Big(h_T^{-1/2}\|v\|_{T^\dagger}
+\|v\|_{T^\dagger}^{1/2}\|\nabla v\|_{T^\dagger}^{1/2}\Big).
\]
\end{lemma}

We shall also use repeatedly the following elementary inequality, which compares the
elasticity tensor with its square: since $(\tr\btau)^2\le d|\btau|^2$ for every symmetric
tensor $\btau$, one checks directly that
\begin{equation}\label{eq:Csquare}
|\CC_i\btau|^2\le(2\mu_i+d\lambda_i)\,(\CC_i\btau,\btau),\qquad
(\CC_i\btau,\btau)=2\mu_i|\btau|^2+\lambda_i(\tr\btau)^2 .
\end{equation}
We write $\|\CC_i^{1/2}\btau\|_S^2:=(\CC_i\btau,\btau)_S$ for brevity.

\subsection{Local energy norm and stability of the local operators}\label{ssec:norm}

For any $T\in\Th$ and $\bhv_T\in\VT$, we define the local discrete energy norm
\begin{multline}\label{eq:tnorm}
\tnorm{\bhv_T}^2_{a_T}:=\sum_{i=1,2}\Big(2\mu_i\|\beps(\bv_{T^i})\|_{T^i}^2
+\lambda_i\|\nabla\!\cdot\!\bv_{T^i}\|_{T^i}^2
+\frac{2\mu_i}{h_T}\|\bv_{(\partial T)^i}-\bv_{T^i}\|_{(\partial T)^i}^2\\
+\frac{d\lambda_i}{h_T}\|{\bf S}(\bhv_{T^i})\|_{(\partial T)^i}^2\Big)
+\big\|\bS_h^{1/2}\jump{\bv_T}_\Gamma\big\|_{T^\Gamma}^2,
\end{multline}
and we set $\tnorm{\bhv_h}_h^2:=\sum_{T\in\Th}\tnorm{\bhv_T}^2_{a_T}$ for $\bhv_h\in\Vh$.

\begin{remark}[On the norm \eqref{eq:tnorm}]\label{rem:norm}
The $\mu$-weighted face contribution involves the full difference
$\bv_{(\partial T)^i}-\bv_{T^i}$, whereas the $\lambda$-weighted one involves the
stabilisation quantity ${\bf S}(\bhv_{T^i})=\bv_{(\partial T)^i}-\bm{\Pi}_{(\partial T)^i}\bv_{T^i}$
only. This distinction, which is the one made in the fitted case \cite{HuangBurman2024}, is
what makes the constants in \Cref{lem:locstab} and in \Cref{thm:stab} independent of the
Lam\'e coefficients; see also \Cref{rem:locking}. Since $\bm{\Pi}_{(\partial T)^i}$ is an
$L^2$-orthogonal projection and $\bv_{(\partial T)^i}$ is a polynomial of degree $k$ on each
subface, one has $\|\bS(\bhv_{T^i})\|_{(\partial T)^i}\le
\|\bv_{(\partial T)^i}-\bv_{T^i}\|_{(\partial T)^i}$, and therefore
\begin{equation}\label{eq:stabnorm}
s_T^{\partial T}(\bhv_T,\bhv_T)\le\tnorm{\bhv_T}_{a_T}^2 .
\end{equation}
\end{remark}

The following elementary estimate will be used twice; it expresses that the distance between
a cell polynomial and the $L^2$-projection of its trace onto the face polynomial spaces is
controlled by its symmetric gradient alone, with a constant that does \emph{not} involve
$\lambda$. The proof is that of \cite[Lem.~4.3]{HuangBurman2024}, the inequalities of
\Cref{lem:disc} being used in place of their fitted counterparts.

\begin{lemma}[Projected trace]\label{lem:proj}
Let $k\ge1$. For any $T\in\Th$, all $i=1,2$ and any $\bv_{T^i}\in\bPk{k+1}(T^i)$,
\begin{equation}\label{eq:proj}
h_T^{-1}\big\|\bv_{T^i}-\bm{\Pi}_{(\partial T)^i}\bv_{T^i}\big\|_{(\partial T)^i}^2
\lesssim\|\beps(\bv_{T^i})\|_{T^i}^2 .
\end{equation}
\end{lemma}

\begin{proof}
Let $\bm{\Pi}_{RM}$ denote the $L^2(T^i)$-orthogonal projection onto $RM$ and set
$\bm\psi:=\bv_{T^i}-\bm{\Pi}_{RM}(\bv_{T^i})$. Since $RM\subset\bPk{1}(T^i)$ and $k\ge1$, the
trace on any subface $F^i$ of a rigid-body motion belongs to $\bPk{k}(F^i)$, so that
$\bm{\Pi}_{(\partial T)^i}\bm{\Pi}_{RM}\bv_{T^i}=\bm{\Pi}_{RM}\bv_{T^i}$ and the left-hand
side of \eqref{eq:proj} is unchanged if $\bv_{T^i}$ is replaced by $\bm\psi$. Using that
$\bm{\Pi}_{(\partial T)^i}$ is an $L^2$-orthogonal projection, then the discrete trace, the
discrete Poincar\'e and the discrete Korn inequalities of \Cref{lem:disc}, which are licit
because $\bm{\psi}$ is orthogonal to $RM$ (hence to the constants) in $L^2(T^i)$, we obtain
\[
h_T^{-1}\|(\bm{I}-\bm{\Pi}_{(\partial T)^i})\bm\psi\|_{(\partial T)^i}^2
\le h_T^{-1}\|\bm\psi\|_{(\partial T)^i}^2
\lesssim h_T^{-2}\|\bm\psi\|_{T^i}^2\lesssim\|\nabla\bm\psi\|_{T^i}^2
\lesssim\|\beps(\bm\psi)\|_{T^i}^2=\|\beps(\bv_{T^i})\|_{T^i}^2,
\]
where the last equality holds because every element of $RM$ has a vanishing symmetric
gradient.
\end{proof}

\begin{lemma}[Local stability and boundedness]\label{lem:locstab}
There are $0<\alpha_\flat\le\alpha_\sharp<+\infty$ such that, for any $T\in\Th$ and any
$\bhv_T\in\VT$,
\begin{equation}\label{eq:locstab}
\alpha_\flat\tnorm{\bhv_T}_{a_T}^2\le\mathcal{N}_T(\bhv_T)\le\alpha_\sharp\tnorm{\bhv_T}_{a_T}^2,
\end{equation}
where
\begin{equation}\label{eq:NT}
\mathcal{N}_T(\bhv_T):=\sum_{i=1,2}\Big(2\mu_i\|\bE_{T^i}(\bhv_T)\|_{T^i}^2
+\lambda_i\|D_{T^i}(\bhv_T)\|_{T^i}^2
+\tfrac{2\mu_i+d\lambda_i}{h_T}\|{\bf S}(\bhv_{T^i})\|_{(\partial T)^i}^2\Big)
+s_T^\Gamma(\bv_T,\bv_T).
\end{equation}
\end{lemma}

\begin{proof}
\emph{Lower bound.} We first bound $\sum_{i=1,2}\|\CC_i^{1/2}\beps(\bv_{T^i})\|_{T^i}^2$. Testing
\eqref{eq:E1}--\eqref{eq:E2} with $\btau=\CC_i\beps(\bv_{T^i})$ and moving the interface
terms to the right-hand side, we obtain
\begin{align*}
\sum_{i=1,2}\|\CC_i^{1/2}\beps(\bv_{T^i})\|_{T^i}^2
&=\sum_{i=1,2}\Big[\big(\bE_{T^i}(\bhv_T),\CC_i\beps(\bv_{T^i})\big)_{T^i}
-\big\langle\bv_{(\partial T)^i}-\bv_{T^i},\CC_i\beps(\bv_{T^i})\bn_T
\big\rangle_{(\partial T)^i}\Big]\\
&\quad+\big\langle(\bI-\bS_h\bK)\bK\bG_{T^1}(\bhv_T)\bn,
\CC_1\beps(\bv_{T^1})\bn\big\rangle_{T^\Gamma}\\
&\quad+\big\langle(\bI-\bS_h\bK)\jump{\bv_T}_\Gamma,
\CC_1\beps(\bv_{T^1})\bn\big\rangle_{T^\Gamma}.
\end{align*}
Write $X:=\big(\sum_{i=1,2}\|\CC_i^{1/2}\beps(\bv_{T^i})\|_{T^i}^2\big)^{1/2}$. The first term is
bounded by $\|\CC_i^{1/2}\bE_{T^i}(\bhv_T)\|_{T^i}X$. In the second one,
$\CC_i\beps(\bv_{T^i})\bn_T$ has its components in $\bPk{k}(F^i)$ on every subface, so that
$\bv_{(\partial T)^i}-\bv_{T^i}$ may be replaced by ${\bf S}(\bhv_{T^i})$; using
\eqref{eq:Csquare} and the discrete trace inequality, it is bounded by
\[
c_{\rm disc}\Big(\tfrac{2\mu_i+d\lambda_i}{h_T}\Big)^{1/2}
\|{\bf S}(\bhv_{T^i})\|_{(\partial T)^i}\,X .
\]
For the third term, we use \eqref{eq:ImSKK}, \eqref{eq:Csquare} and the discrete trace
inequality, together with $\delta=2\mu_1+d\lambda_1$, which give the bound
\[
\tfrac{h_T}{\delta}(2\mu_1+d\lambda_1)c_{\rm disc}^2h_T^{-1}
\|\CC_1^{1/2}\bE_{T^1}(\bhv_T)\|_{T^1}X
=c_{\rm disc}^2\|\CC_1^{1/2}\bE_{T^1}(\bhv_T)\|_{T^1}X .
\]
For the fourth one, we use the factorisation
$\bI-\bS_h\bK=\big(\tfrac{h_T}{\delta}\big)^{1/2}
\big(\tfrac{h_T}{\delta}\bS_h\big)^{1/2}\bS_h^{1/2}$ of \Cref{lem:Sh} together with
$\big|\tfrac{h_T}{\delta}\bS_h\big|\le1$, which yields
\[
\Big(\tfrac{h_T}{\delta}\Big)^{1/2}\|\bS_h^{1/2}\jump{\bv_T}_\Gamma\|_{T^\Gamma}
\|\CC_1\beps(\bv_{T^1})\bn\|_{T^\Gamma}\le c_{\rm disc}
\|\bS_h^{1/2}\jump{\bv_T}_\Gamma\|_{T^\Gamma}X .
\]
Collecting the four bounds and dividing by $X$ shows that $X^2$ is bounded, up to a
constant, by $\mathcal{N}_T(\bhv_T)$, which controls the first two contributions to \eqref{eq:tnorm}.
It remains to bound the face contributions. The $\lambda$-weighted ones are directly part of
$\mathcal{N}_T(\bhv_T)$. For the $\mu$-weighted ones, we write
$\bv_{(\partial T)^i}-\bv_{T^i}={\bf S}(\bhv_{T^i})+(\bm{\Pi}_{(\partial T)^i}\bv_{T^i}-\bv_{T^i})$
and invoke \Cref{lem:proj}, which gives
\[
\frac{2\mu_i}{h_T}\|\bv_{(\partial T)^i}-\bv_{T^i}\|_{(\partial T)^i}^2
\lesssim\frac{2\mu_i}{h_T}\|{\bf S}(\bhv_{T^i})\|_{(\partial T)^i}^2
+2\mu_i\|\beps(\bv_{T^i})\|_{T^i}^2 ,
\]
and both right-hand sides have already been bounded. Note that no $\lambda$ appears in this
last step, which is the reason for the definition \eqref{eq:tnorm}.

\emph{Upper bound.} By \eqref{eq:stabnorm}, the third contribution to
$\mathcal{N}_T(\bhv_T)$ is bounded by $\tnorm{\bhv_T}_{a_T}^2$, and so is $s_T^\Gamma$. For the first
two, we test \eqref{eq:E1}--\eqref{eq:E2} with $\btau=\CC_i\bE_{T^i}(\bhv_T)$ and argue exactly as
above, replacing $\CC_i\beps(\bv_{T^i})$ by $\CC_i\bE_{T^i}(\bhv_T)$ and using
$\sum_{i=1,2}\|\CC_i^{1/2}\bE(\bv_{T^i})\|_{T^i}^2\le a_T(\bhv_T,\bhv_T) $.
\end{proof}

\subsection{Interpolation and approximation}\label{ssec:interp}

Let $\nu\in(\tfrac12,k+1]$ and let $\bm{E}_i:\bm{H}^{1+\nu}(\Omega_i)\to\bm{H}^{1+\nu}(\mathbb{R}^d)$,
$i=1,2$, be stable extension operators. The key operator in the HHO error analysis is the
local interpolation operator $\Iint_T:\bm{H}^{1+\nu}(\Omega_1\cup\Omega_2)\to\VT$ defined by
\begin{equation}\label{eq:interp}
\Iint_T(\bv)=\big(\bm{I}_{T^1}(\bv_1),\bm{I}_{T^2}(\bv_2),
\bm{\Pi}_{(\partial T)^1}(\bv_1),\bm{\Pi}_{(\partial T)^2}(\bv_2)\big)\in\VT,\qquad
\bm{I}_{T^i}(\bv):=\big(\bm{\Pi}^{k+1}_{T^\dagger}(\bm{E}_i(\bv))\big)|_{T^i},
\end{equation}
where $T^\dagger$ is the ball of \Cref{ass:resolve} and $\bm{\Pi}^{k+1}_{T^\dagger}$ is the
$L^2$-orthogonal projection onto $\bPk{k+1}(T^\dagger)$. Projecting on the larger set
$T^\dagger$ rather than on the subcell $T^i$ is what makes the approximation properties
independent of how $\Gamma$ cuts $T$, see \cite[Sect.~3.3]{BCDE2021}. We also define
$\Iint_h(\bv)\in\Vh$ by $\Iint_h(\bv)|_T:=\Iint_T(\bv)$; note that $\Iint_h(\bv)\in\Vhz$
whenever $\bv\in\bm{H}^1_0(\Omega_1\cup\Omega_2)$. The cell component of $\Iint_h$ is denoted
$\bm{I}_\mathcal{T}$.

\begin{lemma}[Local approximation]\label{lem:approx}
Let $\bv\in\bm{H}^1(\Omega_1\cup\Omega_2)$ with $\bv|_{\Omega_i}\in\bm{H}^{m+1}(\Omega_i)$ for
some $m\ge0$ and set $l:=\min\{k+1,m\}$. For any $T\in\Th$ and $i=1,2$,
\begin{equation}\label{eq:approx}
\|\bv-\bm{I}_{T^i}(\bv)\|_{T^i}+h_T^{1/2}\|\bv-\bm{I}_{T^i}(\bv)\|_{(\partial T)^i\cup T^\Gamma}
+h_T|\bv-\bm{I}_{T^i}(\bv)|_{H^1(T^i)}\lesssim h_T^{l+1}|\bm{E}_i(\bv)|_{H^{l+1}(T^\dagger)},
\end{equation}
and, summing over the mesh cells and using \Cref{ass:grading},
$\sum_{T\in\Th}|\bm{E}_i(\bv)|^2_{H^{l+1}(T^\dagger)}\lesssim|\bv|^2_{H^{l+1}(\Omega_i)}$.
\end{lemma}

\Cref{lem:approx} is proved as \cite[(3.6) and (3.7)]{BCDE2021}, combining the
approximation properties of $\bm{\Pi}^{k+1}_{T^\dagger}$ on the ball $T^\dagger$ with
\Cref{lem:mtr}.

In order to handle the nonhomogeneous interface data, it is convenient to introduce the
discrete lifting operator $\bL_{T^1}:\bm{L}^2(T^\Gamma)\to\sPk{k}(T^1)$, defined for any
$\bg\in\bm{L}^2(T^\Gamma)$ by
\begin{equation}\label{eq:lift}
(\bL_{T^1}(\bg),\btau)_{T^1}
+\big\langle(\bI-\bS_h\bK)\bK\,\CC_1\bL_{T^1}(\bg)\bn,\btau\bn\big\rangle_{T^\Gamma}
=\big\langle(\bI-\bS_h\bK)\bg,\btau\bn\big\rangle_{T^\Gamma}\qquad
\forall\,\btau\in\sPk{k}(T^1),
\end{equation}
which is well posed by the argument of \Cref{lem:wp}, and $\bL_{T^i}:=\bm{0}$ for
$T\in\Th^i$, $i=1,2$. Comparing \eqref{eq:E1} and \eqref{eq:lift} shows that, if
$\bg=\jump{\bv}+\bK\bsig(\bv_1)\bn$ on $T^\Gamma$, then
$\bE_{T^1}(\Iint_T\bv)+\bL_{T^1}(\bg)$ is the reconstruction associated with the
\emph{consistent} data. Accordingly, we set
\begin{equation}\label{eq:eta}
\bm\eta_{T^1}(\bv):=\bE_{T^1}(\Iint_T\bv)+\bL_{T^1}\big(\jump{\bv}+\bK\bsig(\bv_1)\bn\big)
-\beps(\bv_1),\qquad
\bm\eta_{T^2}(\bv):=\bE_{T^2}(\Iint_T\bv)-\beps(\bv_2).
\end{equation}

\begin{lemma}[Approximation of the reconstructions]\label{lem:etaapprox}
Let $\bv\in\bm{H}^{k+2}(\Omega_1\cup\Omega_2)$. For any $T\in\Th$ and $i=1,2$,
\begin{equation}\label{eq:etaapprox}
\|\bm\eta_{T^i}(\bv)\|_{T^i}+h_T^{1/2}\|\bm\eta_{T^i}(\bv)\|_{(\partial T)^i\cup T^\Gamma}
\lesssim h_T^{k+1}\sum_{j=1,2}|\bm{E}_j(\bv)|_{H^{k+2}(T^\dagger)} .
\end{equation}
(The right-hand side of \eqref{eq:etaapprox} only involves $\bm{E}_2(\bv)$ when $i=2$.)
\end{lemma}

\begin{proof}
We only treat $i=1$, the case $i=2$ being simpler. Set $\bm\eta'_{T^1}(\bv)=\bE_{T^1}(\Iint_T\bv)+\bL_{T^1}\big(\jump{\bv}+\bK\bsig(\bv_1)\bn\big)
	-\beps (\bm{I}_{T^1}\bv_1)$. Testing the equation satisfied by
$\bm\eta'_{T^1}(\bv)$ with $\btau=\bm\eta'_{T^1}(\bv)$ and using the definitions
\eqref{eq:E1} and \eqref{eq:lift} together with
$\jump{\bv}+\bK\bsig(\bv_1)\bn-\jump{\Iint_T\bv}_\Gamma
=\jump{\bv-\bm{I}_T\bv}_\Gamma+\bK\bsig(\bv_1)\bn$, we obtain
\begin{align*}
\|\bm\eta'_{T^1}(\bv)\|_{T^1}^2
&+\big\langle(\bI-\bS_h\bK)\bK\CC_1\bm\eta'_{T^1}(\bv)\bn,\bm\eta'_{T^1}(\bv)\bn
\big\rangle_{T^\Gamma}\\
&=\big\langle\bm{\Pi}_{(\partial T)^1}\bv-\bm{I}_{T^1}\bv,\bm\eta'_{T^1}(\bv)\bn_T
\big\rangle_{(\partial T)^1}\\
&\quad+\big\langle(\bI-\bS_h\bK)\jump{\bv-\bm{I}_T\bv}_\Gamma,\bm\eta'_{T^1}(\bv)\bn
\big\rangle_{T^\Gamma}\\
&\quad+\big\langle(\bI-\bS_h\bK)\bK\bsig\big(\bv_1
-\bm{I}_{T^1}\bv\big)\bn,\bm\eta'_{T^1}(\bv)\bn\big\rangle_{T^\Gamma}.
\end{align*} 
Since $\big(\bm{\eta}'_{T^1}(\bv)\bn\big)|_{(\partial T)^1}\in {\bf P}^k(\mathcal{F}_{(\partial T)^1})$ owing to the function spaces on the reconstruction operator $\bE_{T^1}$ and the lift operator $\bL_{T^1}$, we can replace $\bm{\Pi}_{(\partial T)^1}\bv$ by $\bv$ in the first term on the right hand side. The second term on the left-hand side is
nonnegative. Bounding the three terms of the right-hand side by the Cauchy--Schwarz
inequality, using \eqref{eq:unitbound} and \eqref{eq:ImSKK} together with
the discrete trace inequality for $\|\bm\eta'_{T^1}(\bv)\|_{(\partial T)^1\cup T^\Gamma}$, and dividing by
$\|\bm\eta'_{T^1}(\bv)\|_{T^1}$ leads to
\begin{align*}
	  \|\bm\eta'_{T^1}(\bv)\|_{T^1}
	  &+h_T^{1/2}\|\bm\eta'_{T^1}(\bv)\|_{(\partial T)^1\cup T^\Gamma}
	\lesssim \|\bm\eta'_{T^1}(\bv)\|_{T^1} \\
	\lesssim& h_T^{-\frac12} \big[\|\bv-\bm{I}_{T^1}\bv\|_{(\partial T)^1} +\|\jump{\bv-\bm{I}_T\bv}_\Gamma\|_{T^\Gamma}
	+\tfrac{h_T}{\delta}\|\bsig(\bv_1-\bm{I}_{T^1}\bv)\bn\|_{T^\Gamma}\big].
\end{align*} 
Invoking \Cref{lem:approx} together with $\bm{\eta}_{T^1}(\bv)=\bm{\eta}'_{T^1}(\bv) +\beps\big( \bm{I}_{T^1}(\bv_1)-\bv_1\big)$ and the triangular inequality then proves \eqref{eq:etaapprox}. 
\end{proof}

We close this section with a discrete Poincar\'e inequality on $\Vhz$, which is the analogue
of \cite[Lem.~6.2]{BurmanDuranErn2022CMAME}; we set $\Omega^\Gamma:=\bigcup_{T\in\Th^\Gamma}T$
and denote by $l_\Omega$ the diameter of $\Omega$.

\begin{lemma}[Discrete Poincar\'e inequality]\label{lem:poincare}
For any $\bhv_h\in\Vhz$ we have
\[
\|\bv_\mathcal{T}\|_\Omega\lesssim l_\Omega\,\mu_1^{-1/2}\,\tnorm{\bhv_h}_h .
\]
\end{lemma}

\begin{proof}
Let $\bm{\mathcal{X}}\in\bm{\mathcal{H}}^1(\Omega)$ be such that
$-\nabla\!\cdot\!\bm{\mathcal{X}}=\bv_\mathcal{T}$ and
$\|\bm{\mathcal{X}}\|_{\bm{\mathcal{H}}^1(\Omega)}\lesssim l_\Omega\|\bv_\mathcal{T}\|_\Omega$.
Integrating by parts cellwise and using that $\bm{\mathcal{X}}$ is single-valued at the mesh
interfaces and that $\bv_\mathcal{F}$ vanishes on the boundary faces, we obtain
\[
\|\bv_\mathcal{T}\|_\Omega^2=\sum_{T\in\Th}\sum_{i=1,2}\Big(
(\nabla\bv_{T^i},\bm{\mathcal{X}})_{T^i}
+(\bv_{(\partial T)^i}-\bv_{T^i},\bm{\mathcal{X}}\bn_T)_{(\partial T)^i}\Big)
-\sum_{T\in\Th^\Gamma}(\jump{\bv_T}_\Gamma,\bm{\mathcal{X}}\bn)_{T^\Gamma}.
\]
Invoking the Cauchy--Schwarz inequality, the definition \eqref{eq:tnorm} of the norm,
\eqref{eq:Shbound} for the interface term, \Cref{lem:mtr} followed by Young's inequality for
the face terms, and \Cref{ass:grading}, we infer
$\|\bv_\mathcal{T}\|_\Omega^2\lesssim\mu_1^{-1/2}\tnorm{\bhv_h}_h
\|\bm{\mathcal{X}}\|_{\bm{\mathcal{H}}^1(\Omega)}$, and we conclude by the stability of
$\bm{\mathcal{X}}$.
\end{proof}

\section{Stability and error analysis}\label{sec:analysis}

\subsection{Stability of the discrete bilinear form}\label{ssec:stab}

\begin{theorem}[Stability]\label{thm:stab}
The discrete bilinear form $a_h$ defined in \eqref{eq:ah} is symmetric, and there are
$0<\alpha_\flat\le\alpha_\sharp<+\infty$, independent of $h$, of the Lam\'e coefficients and
of $\bK$, such that
\begin{equation}\label{eq:coer}
\alpha_\flat\tnorm{\bhv_h}_h^2\le a_h(\bhv_h,\bhv_h)\le\alpha_\sharp\tnorm{\bhv_h}_h^2,
\qquad\forall\,\bhv_h\in\Vh .
\end{equation}
Moreover, $\tnorm{\cdot}_h$ is a norm on $\Vhz$, so that the space semi-discrete problem
\eqref{eq:semi} is well posed.
\end{theorem}

\begin{proof}
Symmetry is clear from \eqref{eq:aT} and \eqref{eq:stab}, since $\bS_h$ and
$(\bI-\bS_h\bK)\bK$ are symmetric. Let $\bhv_h\in\Vh$ and $T\in\Th$. Since
$(\bI-\bS_h\bK)\bK$ is positive semi-definite by \Cref{lem:Sh}, we have
\[
a_T(\bhv_T,\bhv_T)+s_T(\bhv_T,\bhv_T)\ge\mathcal{N}_T(\bhv_T),
\]
and the lower bound in \eqref{eq:coer} follows from \Cref{lem:locstab} after summation over
$T\in\Th$. For the upper bound, observe that, by \eqref{eq:ImSKK},
\eqref{eq:Csquare} and the discrete trace inequality,
\[
\big\langle(\bI-\bS_h\bK)\bK\bG_{T^1}(\bhv_T)\bn,\bG_{T^1}(\bhv_T)\bn\big\rangle_{T^\Gamma}
\le\frac{h_T}{\delta}\|\CC_1\bE_{T^1}(\bhv_T)\bn\|_{T^\Gamma}^2
\le c_{\rm disc}^2\|\CC_1^{1/2}\bE_{T^1}(\bhv_T)\|_{T^1}^2,
\]
so that $\mathcal{N}_T(\bhv_T)$ again controls $a_T+s_T$ from above and the upper bound in \eqref{eq:coer} holds by \Cref{lem:locstab}. Finally,
if $\bhv_h\in\Vhz$ satisfies $\tnorm{\bhv_h}_h=0$, then $\beps(\bv_{T^i})=\bm{0}$ on each
subcell and $\bv_\mathcal{T}=\bm{0}$ by \Cref{lem:poincare}; the face components then vanish
as well, because $\bv_{(\partial T)^i}-\bv_{T^i}=\bm{0}$ on $(\partial T)^i$.
\end{proof}

\subsection{Consistency}\label{ssec:consistency}

Let $\bu\in C^0(\overline J;\bm{H}^1_0(\Omega_1\cup\Omega_2))\cap C^1(\overline J;
\bm{L}^2(\Omega))$ be the solution of \eqref{eq:weak} and set $\bu_i=\bu|_{\Omega_i}$. For
$t\in\overline J$ we define the discrete error
\begin{equation}\label{eq:err}
\bhe_h(t):= \Iint_h(\bu(t))-\bhu_h(t) ,\qquad
\bm{e}_\mathcal{T}(t):=\bm{I}_\mathcal{T}(\bu(t)) -\bu_\mathcal{T}(t),
\end{equation}
and, for a smooth field $\bv$, the consistency seminorm
\begin{multline}\label{eq:sharp}
|\bv|_{\sharp h}^2:=\sum_{T\in\Th}\bigg\{\sum_{i=1,2}\Big[
\|\CC_i^{1/2}\bm\eta_{T^i}(\bv)\|_{T^i}^2
+h_T\|\CC_i^{1/2}\bm\eta_{T^i}(\bv)\|_{(\partial T)^i\cup T^\Gamma}^2\\
+\frac{h_T}{2\mu_i}\big\|(\bm{I}-\bm{\Pi}_{(\partial T)^i})\big(\bsig(\bv_i)\bn_T\big)
\big\|_{(\partial T)^i}^2
+\frac{2\mu_i+d\lambda_i}{h_T}
\big\|\bm{\Pi}_{(\partial T)^i}(\bv_i-\bm{I}_{T^i}\bv)\big\|_{(\partial T)^i}^2\Big]\\
+\big\|\bS_h^{1/2}\big(\jump{\bm{I}_T\bv}_\Gamma-\jump{\bv}\big)\big\|_{T^\Gamma}^2\bigg\},
\end{multline}
with $\bm\eta_{T^i}(\bv)$ defined in \eqref{eq:eta}. Finally, for $t\in J$ and
$\bhv_h\in\Vhz$ we introduce the consistency error
\begin{equation}\label{eq:Xi}
\Xi_h(\bu(t);\bhv_h):=a_h\big(\Iint_h(\bu(t)),\bhv_h\big)-\ell_h(\bhv_h)
+\big(\varrho\,\dtt\bu(t),\bv_\mathcal{T}\big)_\Omega .
\end{equation}

\begin{lemma}[Consistency]\label{lem:consistency}
Assume that $\bu(t)|_{\Omega_i}\in\bm{H}^{1+s}(\Omega_i)$ for some $s>\tfrac12$, $i=1,2$.
Then, for any $\bhv_h\in\Vhz$,
\begin{equation}\label{eq:consist}
\big|\Xi_h(\bu(t);\bhv_h)\big|\lesssim|\bu(t)|_{\sharp h}\,\tnorm{\bhv_h}_h .
\end{equation}
\end{lemma}

\begin{proof}
Using the representation \eqref{eq:aTcomp} of $a_T$, the definition \eqref{eq:lT} of $\ell_T$
and the definition \eqref{eq:lift} of the lifting, and recalling that
$\bg_D=\jump{\bu}+\bK\bsig(\bu_1)\bn$ and $\bg_N=\jump{\bsig(\bu)\bn}$ on $\Gamma$, we may
write $\Xi_h(\bu;\bhv_h)=\Psi+\sum_{T\in\Th}\Phi_T$, where $\Phi_T$ collects the terms
produced by the lifting, i.e. $\Phi_T$ is obtained from \eqref{eq:aTcomp} upon replacing
$\bE_{T^1}(\Iint_T\bu)$ by $\bL_{T^1}(\bg_D)$. Integrating by parts the volume term
$(\nabla\!\cdot\!\bsig(\bu),\bv_{T^i})_{T^i}$ coming from the momentum equation, using the
definitions \eqref{eq:E1}--\eqref{eq:E2}, the single-valuedness of $\bsig(\bu)\bn$ across the
mesh interfaces not lying on $\Gamma$ and the two interface conditions, all the terms
involving $\bg_D$ and $\bg_N$ cancel and one is left with
\begin{multline}\label{eq:Xisplit}
\Xi_h(\bu;\bhv_h)=\sum_{T\in\Th}\bigg\{\sum_{i=1,2}\Big[
\big(\CC_i\bm\eta_{T^i}(\bu),\beps(\bv_{T^i})\big)_{T^i}
+\big\langle\CC_i\bm\eta_{T^i}(\bu)\bn_T,
\bv_{(\partial T)^i}-\bv_{T^i}\big\rangle_{(\partial T)^i}\Big]\\
-\big\langle(\bI-\bS_h\bK)\CC_1\bm\eta_{T^1}(\bu)\bn,\jump{\bv_T}_\Gamma
\big\rangle_{T^\Gamma}
+\big\langle\bS_h\big(\jump{\bm{I}_T\bu}_\Gamma-\jump{\bu}\big),\jump{\bv_T}_\Gamma
\big\rangle_{T^\Gamma}
+s_T^{\partial T}\big(\Iint_T\bu,\bhv_T\big)\bigg\}.
\end{multline}
We bound the five contributions in turn, writing $\Lambda_i:=2\mu_i+d\lambda_i$ for brevity.
The first one is bounded by
$\|\CC_i^{1/2}\bm\eta_{T^i}(\bu)\|_{T^i}\|\CC_i^{1/2}\beps(\bv_{T^i})\|_{T^i}$. In the second
one, we insert $\bm{\Pi}_{(\partial T)^i}\big(\CC_i\bm\eta_{T^i}(\bu)\bn_T\big)$ and split
\begin{multline*}
\big\langle\CC_i\bm\eta_{T^i}(\bu)\bn_T,\bv_{(\partial T)^i}-\bv_{T^i}\big\rangle_{(\partial T)^i}
=\big\langle\bm{\Pi}_{(\partial T)^i}\big(\CC_i\bm\eta_{T^i}(\bu)\bn_T\big),
{\bf S}(\bhv_{T^i})\big\rangle_{(\partial T)^i}\\
+\big\langle(\bm{I}-\bm{\Pi}_{(\partial T)^i})\big(\CC_i\bm\eta_{T^i}(\bu)\bn_T\big),
\bv_{(\partial T)^i}-\bv_{T^i}\big\rangle_{(\partial T)^i}.
\end{multline*}
Since $\CC_i\bm\eta_{T^i}(\bu)\bn_T=\CC_i\big(\bE_{T^i}(\Iint_T\bu)
+\bL_{T^i}(\bg_D)\big)\bn_T-\bsig(\bu_i)\bn_T$ and the first contribution has its components
in $\bPk{k}(F^i)$ on every subface, the second term above equals
$-\big\langle(\bm{I}-\bm{\Pi}_{(\partial T)^i})(\bsig(\bu_i)\bn_T),
\bv_{(\partial T)^i}-\bv_{T^i}\big\rangle_{(\partial T)^i}$, which is bounded by the third
contribution to \eqref{eq:sharp} times
$(2\mu_i/h_T)^{1/2}\|\bv_{(\partial T)^i}-\bv_{T^i}\|_{(\partial T)^i}\le\tnorm{\bhv_T}_{a_T}$.
The first term above is bounded, using \eqref{eq:Csquare}, by
\begin{multline}\label{eq:cons2}
\Big(\tfrac{h_T}{\Lambda_i}\Big)^{1/2}\|\CC_i\bm\eta_{T^i}(\bu)\bn_T\|_{(\partial T)^i}
\cdot\Big(\tfrac{\Lambda_i}{h_T}\Big)^{1/2}\|{\bf S}(\bhv_{T^i})\|_{(\partial T)^i}\\
\le h_T^{1/2}\|\CC_i^{1/2}\bm\eta_{T^i}(\bu)\|_{(\partial T)^i}
\cdot\Big(\tfrac{\Lambda_i}{h_T}\Big)^{1/2}\|{\bf S}(\bhv_{T^i})\|_{(\partial T)^i},
\end{multline}
and the last factor in \eqref{eq:cons2} is bounded by $\tnorm{\bhv_T}_{a_T}$, because
$\tfrac{2\mu_i}{h_T}\|{\bf S}(\bhv_{T^i})\|^2\le\tfrac{2\mu_i}{h_T}
\|\bv_{(\partial T)^i}-\bv_{T^i}\|^2$ by \eqref{eq:stabnorm}, while the $\lambda$-weighted
part is a term of \eqref{eq:tnorm}. For the third contribution we use the factorisation
$\bI-\bS_h\bK=\big(\tfrac{h_T}{\delta}\big)^{1/2}
\big(\tfrac{h_T}{\delta}\bS_h\big)^{1/2}\bS_h^{1/2}$ of \Cref{lem:Sh}, which gives, using
again \eqref{eq:Csquare} and $\delta=\Lambda_1$,
\[
\Big(\tfrac{h_T}{\delta}\Big)^{1/2}\|\CC_1\bm\eta_{T^1}(\bu)\bn\|_{T^\Gamma}
\|\bS_h^{1/2}\jump{\bv_T}_\Gamma\|_{T^\Gamma}\le
h_T^{1/2}\|\CC_1^{1/2}\bm\eta_{T^1}(\bu)\|_{T^\Gamma}
\|\bS_h^{1/2}\jump{\bv_T}_\Gamma\|_{T^\Gamma}.
\]
The fourth contribution is bounded by
$\|\bS_h^{1/2}(\jump{\bm{I}_T\bu}_\Gamma-\jump{\bu})\|_{T^\Gamma}
\|\bS_h^{1/2}\jump{\bv_T}_\Gamma\|_{T^\Gamma}$. Finally, since
${\bf S}(\Iint_T\bu)=\bm{\Pi}_{(\partial T)^i}(\bu_i-\bm{I}_{T^i}\bu)$, the Cauchy--Schwarz
inequality for $s_T^{\partial T}$ together with \eqref{eq:stabnorm} bounds the last
contribution by the fourth term of \eqref{eq:sharp} times $\tnorm{\bhv_T}_{a_T}$. Summing over
$T\in\Th$ and using the Cauchy--Schwarz inequality gives \eqref{eq:consist}.
\end{proof}

\begin{lemma}[Approximation of the consistency seminorm]\label{lem:sharpapprox}
Let $\bv\in\bm{H}^{k+2}(\Omega_1\cup\Omega_2)$. Then
\begin{equation}\label{eq:sharpapprox}
|\bv|_{\sharp h}\lesssim h^{k+1}\,\Phi_k(\bv),\qquad
\Phi_k(\bv):=\sum_{i=1,2} \mu_i^{1/2} (1+\kappa_i)
\,|\bv|_{\bm{H}^{k+2}(\Omega_i)} .
\end{equation}
\end{lemma}

\begin{proof}
The first two terms of \eqref{eq:sharp} are bounded by
\[
(2\mu_i+d\lambda_i)\Big(\|\bm\eta_{T^i}(\bv)\|_{T^i}^2
+h_T\|\bm\eta_{T^i}(\bv)\|_{(\partial T)^i\cup T^\Gamma}^2\Big)
\lesssim(2\mu_i+d\lambda_i)h_T^{2(k+1)}\sum_{j=1,2}|\bm{E}_j(\bv)|_{H^{k+2}(T^\dagger)}^2
\]
thanks to \Cref{lem:etaapprox}. For the third one, we use the approximation properties of
$\bm{\Pi}_{(\partial T)^i}$ together with \Cref{lem:mtr} and
$|\bsig(\bv_i)|_{H^{k+1}(\Omega_i)}\le(2\mu_i+d\lambda_i)|\bv|_{H^{k+2}(\Omega_i)}$, which
gives the bound $\tfrac{(2\mu_i+d\lambda_i)^2}{2\mu_i}h_T^{2(k+1)}
|\bm{E}_i(\bv)|_{H^{k+2}(T^\dagger)}^2$. The fourth one is bounded by
$(2\mu_i+d\lambda_i)h_T^{-1}\|\bv_i-\bm{I}_{T^i}\bv\|_{(\partial T)^i}^2
\lesssim(2\mu_i+d\lambda_i)h_T^{2(k+1)}|\bm{E}_i(\bv)|_{H^{k+2}(T^\dagger)}^2$
\Cref{lem:approx}. The last one is bounded, using \eqref{eq:Shbound}, by
$\tfrac{\delta}{h_T}\|\jump{\bm{I}_T\bv}_\Gamma-\jump{\bv}\|_{T^\Gamma}^2
\lesssim\delta\,h_T^{2(k+1)}\sum_{j=1,2}|\bm{E}_j(\bv)|_{H^{k+2}(T^\dagger)}^2$, again by
\Cref{lem:approx}. Summing over the cells and invoking \Cref{ass:grading} gives
\eqref{eq:sharpapprox}, the dominant contribution being the third one.
\end{proof}

\begin{remark}[On the Lam\'e dependence and locking]\label{rem:locking}
The factor $\Phi_k(\bv)= \sum_{i=1,2}\mu_i^{1/2} (1+\kappa_i)|\bv|_{H^{k+2}(\Omega_i)}$
in \eqref{eq:sharpapprox} degenerates as $\kappa_i=\lambda_i/\mu_i\to+\infty$. A sharper
analysis, splitting $\CC_i$ into its $\mu$- and $\lambda$-parts as in
\cite[Thm.~5.2 and Lem.~5.1]{HuangBurman2024}, replaces
$(2\mu_i+d\lambda_i)|\bv|_{H^{k+2}(\Omega_i)}$ by
$\mu_i|\bv|_{H^{k+2}(\Omega_i)}+\lambda_i|\nabla\!\cdot\!\bv|_{H^{k+1}(\Omega_i)}$ in the
first three terms of \eqref{eq:sharp}, and the resulting estimate is then uniform in
$\lambda$ whenever $\lambda|\nabla\!\cdot\!\bu|$ stays bounded, i.e.\ the method would be
locking free. This is however \emph{not} the case for the fourth term of \eqref{eq:sharp},
which carries the intrinsic factor $(2\mu_i+d\lambda_i)^{1/2}h^{k+1}$. In the
fitted case \cite{HuangBurman2024}, this obstruction is removed by weighting the stabilisation
by the anisotropic matrix $h_T^{-1}(2\mu_i\bI+d\lambda_i\bn_T\otimes\bn_T)$ and by measuring
the interpolation error with an $H(\dive)$-conforming (BDM) interpolate, whose normal trace
error vanishes on every face. Such an interpolate is not available on the cut subcells, whose
boundary comprises the curved piece $T^\Gamma$, and we therefore do not claim locking-free
behaviour here. The numerical experiments of \Cref{ssec:incomp} show a growth of the errors
proportional to $\lambda^{1/2}$ with the stabilisation \eqref{eq:stab}, and confirm that the
anisotropic weight restores the uniformity of the $L^2$-error in practice.
\end{remark}

\subsection{Energy-error estimate}\label{ssec:energy}

For $\bhv_h\in C^0(\overline J;\Vhz)$ and $t\in J$, we set
$\|\bhv_h\|_{C^0(0,t;\,h)}:=\sup_{s\in[0,t]}\tnorm{\bhv_h(s)}_h$.

\begin{proposition}[Energy-error estimate]\label{prop:energy}
Let $\bu$ solve \eqref{eq:weak} and let $\bhu_h$ solve \eqref{eq:semi} with the initial
conditions \eqref{eq:ic}. Assume $\bu|_{\Omega_i}\in C^2(\overline J;\bm{H}^{1+\nu}(\Omega_i))$
with $\nu\in(\tfrac12,k+1]$, $i=1,2$. Then, for all $t\in J$,
\begin{multline}\label{eq:energyest}
\big\|\varrho^{1/2}\big(\dt\bu_\mathcal{T}-\bm{I}_\mathcal{T}(\dt\bu)\big)
\big\|_{C^0(0,t;\bm{L}^2(\Omega))}+\big\|\bhu_h-\Iint_h\bu\big\|_{C^0(0,t;\,h)}\\
\lesssim|\bu|_{C^0(0,t;\sharp h)}+|\dt\bu|_{L^1(0,t;\sharp h)},
\end{multline}
where
\begin{eqnarray*}
|\bu|_{C^0(0,t;\sharp h)}:=\sup_{s\in[0,t]}\Big( \|\bu|_{\natural h}
+\Big\|\sqrt{\frac{\varrho}{\mu_1}}\Big( \partial_{tt}\bu(s) -\bm{I}_{\cal T}(\partial_{tt}\bu(s)) \Big)\Big\|_{\Omega^{\Gamma}}\Big),\\
|\dt\bu|_{L^1(0,t;\ \natural h)}:=\int_0^t\Big( |\dt\bu|_{\natural h}
+\Big\|\sqrt{\frac{\varrho}{\mu_1}}\Big( \partial_{ttt}\bu(s) -\bm{I}_{\cal T}(\partial_{ttt}\bu(s)) \Big)\Big\|_{\Omega^{\Gamma}} \Big)ds.
\end{eqnarray*}
Consequently, if $\bu|_{\Omega_i}\in C^3(\overline J;\bm{H}^{k+2}(\Omega_i))$,
\begin{equation}\label{eq:energyrate}
\big\|\varrho^{1/2}\big(\dt\bu_\mathcal{T}-\dt\bu\big)
\big\|_{C^0(0,t;\bm{L}^2(\Omega))}+\big\|\bhu_h-\bu\big\|_{C^0(0,t;\,h)}
\lesssim h^{k+1} \Phi(t,\bu),
\end{equation}
where $\Phi(t,\bu)=|\bu|_{L^{\infty}(0,t;\,\bm{H}^{k+2}(\Omega_1\cup\Omega_2))}
+|\dt\bu|_{L^1(0,t;\,\bm{H}^{k+2}(\Omega_1\cup\Omega_2))}
+|\partial_{tt}\bu|_{L^{\infty}(0,t;\,\bm{H}^{k+1}(\Omega_1\cup\Omega_2))}
+|\partial_{ttt}\bu|_{L^1(0,t;\,\bm{H}^{k+1}(\Omega_1\cup\Omega_2))}$.
\end{proposition}

\begin{proof}
\emph{Step 1: error equation.} Subtracting \eqref{eq:Xi} from \eqref{eq:semi} we obtain, for
all $\bhw_h\in\Vhz$ and all $t\in J$,
\begin{equation}\label{eq:erreq}
\big(\varrho\,\dtt\bm{e}_\mathcal{T}(t),\bw_\mathcal{T}\big)_\Omega
+a_h\big(\bhe_h(t),\bhw_h\big)=\Theta_h(t;\bhw_h),
\end{equation}
with
\begin{equation}\label{eq:Theta}
\Theta_h(t;\bhw_h):=\big(\varrho\big(\dtt\bm{I}_\mathcal{T}\bu-\dtt\bu\big),
\bw_\mathcal{T}\big)_{\Omega^{\Gamma}}+\Xi_h(\bu(t);\bhw_h),
\end{equation}
where we used that $\dtt\bm{I}_\mathcal{T}\bu=\bm{I}_\mathcal{T}(\dtt\bu)$ and that
$(\varrho(\bm{I}_\mathcal{T}(\dtt\bu)-\dtt\bu),\bw_\mathcal{T})_{T^i}=0$ for every uncut cell,
because $\bm{I}_{T^i}$ coincides there with the $L^2$-projection onto $\bPk{k+1}(T)$. Hence,
by the Cauchy--Schwarz inequality, \Cref{lem:poincare} and \Cref{lem:consistency},
\begin{equation}\label{eq:Thetabound}
\|\Theta_h(t;\cdot)\|_{(\Vhz)'}\lesssim
\Big\|\sqrt{\frac{\varrho}{\mu_1}}\big(\dtt\bm{I}_\mathcal{T}\bu-\dtt\bu\big)\Big\|_{\Omega^\Gamma}
+|\bu(t)|_{\sharp h}.
\end{equation}

\emph{Step 2: stability argument.} Testing \eqref{eq:erreq} with
$\bhw_h=\dt\bhe_h(t)$, using the symmetry of $a_h$ and integrating by parts in time on the
right-hand side gives
\[
\frac{d}{dt}\Big\{\tfrac12\big\|\varrho^{1/2}\dt\bm{e}_\mathcal{T}(t)\big\|_\Omega^2
+\tfrac12a_h(\bhe_h(t),\bhe_h(t))\Big\}
=\frac{d}{dt}\Theta_h\big(t;\bhe_h(t)\big)-\dot\Theta_h\big(t;\bhe_h(t)\big),
\]
with $\dot\Theta_h(t;\cdot):=(\varrho(\partial_{ttt}\bm{I}_\mathcal{T}\bu-\partial_{ttt}\bu),
\cdot)_{\Omega^\Gamma}+\Xi_h(\dt\bu(t);\cdot)$, which satisfies the analogue of
\eqref{eq:Thetabound} with $\dt\bu$ in place of $\bu$. Integrating in time from $0$ to $t$,
using the coercivity \eqref{eq:coer} and reasoning exactly as in the proof of
\cite[Thm.~1]{BDES2021} — that is, bounding the right-hand side by Young's and H\"older's
inequalities and using that the left-hand side evaluated at any $t'\in(0,t)$ is bounded by
the right-hand side — we obtain
\[
\tfrac12\big\|\varrho^{1/2}\dt\bm{e}_\mathcal{T}\big\|_{C^0(0,t;\bm{L}^2(\Omega))}^2
+\tfrac18\alpha_\flat\|\bhe_h\|_{C^0(0,t;\,h)}^2
\lesssim\|\Theta_h\|_{C^0(0,t;(\Vhz)')}^2+\|\dot\Theta_h\|_{L^1(0,t;(\Vhz)')}^2
+\tnorm{\bhe_h(0)}_h^2 .
\]
with $\|\Theta_h\|_{C^0(0,t;\ (\Vhz)')}:=\sup_{s\in[0,t]} \|\Theta_h(s;\cdot)\|_{(\Vhz)'}$
and $\|\dot{\Theta}_h\|_{L^1(0,t;\ (\Vhz)')}:=\int_0^t \|\dot{\Theta}_h(s;\cdot)\|_{(\Vhz)'}ds$. Still reasoning as in this proof (where one uses $\bm{e}_{\cal T}(0)=\bm{0}$) shows that $\tnorm{\bhe_h(0)}_h\lesssim \tnorm{\bu_0}_h$. Combining this bound with the estimate \eqref{eq:Thetabound} proves \eqref{eq:energyest}. 

\emph{Step 3: convergence rates.} The estimate \eqref{eq:energyrate} follows from \eqref{eq:energyest} and the triangle inequality, after invoking the approximation properties in \Cref{lem:consistency} and those of the gradient reconstruction from \Cref{lem:sharpapprox}.
\end{proof}

\subsection{An elliptic projection and the \texorpdfstring{$L^2$}{L2}-error estimate}
\label{ssec:L2}

To obtain an $L^2$-error estimate of order $h^{k+2}$, we use a Ritz-type projection
\cite{Wheeler1973,Baker1976}, defined as follows: given $\bu$ solving \eqref{eq:model} at a
fixed time, find $\hat{\bm{R}}_h\bu\in\Vhz$ such that, for all $\bhw_h\in\Vhz$,
\begin{equation}\label{eq:ritz}
\begin{aligned}
a_h(\hat{\bm{R}}_h\bu,\bhw_h)&=\big(-\nabla\!\cdot\!\bsig(\bu),\bw_\mathcal{T}\big)_\Omega
-\big\langle(\bI-\bS_h\bK)\bg_D,\bG_{\mathcal{T}^1}(\bhw_h)\bn\big\rangle_\Gamma\\
&\quad+\big\langle\bS_h\bg_D,\jump{\bw_\mathcal{T}}\big\rangle_\Gamma
+\langle\bg_N,\bw_{\mathcal{T}^2}\rangle_\Gamma .
\end{aligned}
\end{equation}
In other words, $\hat{\bm{R}}_h\bu$ is the unfitted HHO solution of the \emph{stationary}
elasticity problem with the same data, which is well posed by \Cref{thm:stab}. Since the
right-hand side of \eqref{eq:ritz} is exactly $\ell_h(\bhw_h)$ with $\bff$ replaced by
$-\nabla\!\cdot\!\bsig(\bu)$, subtracting \eqref{eq:ritz} from
$a_h(\Iint_h\bu,\bhw_h)$ and invoking \eqref{eq:Xisplit} gives the identity
\begin{equation}\label{eq:ritzconsist}
a_h\big(\Iint_h\bu-\hat{\bm{R}}_h\bu,\bhw_h\big)
=\Xi_h^{\rm ell}(\bu;\bhw_h),\qquad\forall\,\bhw_h\in\Vhz,
\end{equation}
where $\Xi_h^{\rm ell}$ is the right-hand side of \eqref{eq:Xisplit}. \Cref{thm:stab},
\Cref{lem:consistency} and \Cref{lem:sharpapprox} therefore yield at once the energy estimate
for the elliptic projection.

\begin{theorem}[Energy estimate for the elliptic projection]\label{thm:ritzenergy}
Let $k\ge1$ and assume $\bu\in\bm{H}^{k+2}(\Omega_1\cup\Omega_2)$. Then
\begin{equation}\label{eq:ritzenergy}
\tnorm{\Iint_h\bu-\hat{\bm{R}}_h\bu}_h\lesssim|\bu|_{\sharp h}\lesssim h^{k+1}\Phi_k(\bu).
\end{equation}
\end{theorem}

To improve on \eqref{eq:ritzenergy} in the $L^2$-norm we use a duality argument and we assume
that the auxiliary problem
\begin{equation}\label{eq:dual}
\left\{
\begin{aligned}
-\nabla\!\cdot\!\bsig(\bzeta)&=\bvt&&\text{in }\Omega_1\cup\Omega_2,\\
\jump{\bzeta}+\bK\bsig(\bzeta_1)\bn&=\bm{0},\qquad\jump{\bsig(\bzeta)\bn}=\bm{0}
&&\text{on }\Gamma,\\
\bzeta&=\bm{0}&&\text{on }\partial\Omega,
\end{aligned}
\right.
\end{equation}
has a unique solution for any $\bvt\in\bm{L}^2(\Omega)$ and that it satisfies the elliptic
regularity \cite{Bacuta2003}
\begin{equation}\label{eq:reg}
\|\bzeta\|_{\bm{H}^2(\Omega_1\cup\Omega_2)}
\lesssim\|\bvt\|_\Omega.
\end{equation}
Problem
\eqref{eq:dual} is the adjoint of \eqref{eq:model}, which is self-adjoint, so that the same
interface conditions appear.

\begin{theorem}[$L^2$-estimate for the elliptic projection]\label{thm:ritzL2}
Let $k\ge1$ and assume $\bu\in\bm{H}^{k+2}(\Omega_1\cup\Omega_2)$ and \eqref{eq:reg}. Then
\begin{equation}\label{eq:ritzL2}
\big\|\bm{I}_\mathcal{T}\bu-\bm{R}_\mathcal{T}\bu\big\|_\Omega
\lesssim h^{k+2}\sum_{i=1,2}\mu_i\,(1+\kappa_i)^2
\,|\bu|_{\bm{H}^{k+2}(\Omega_i)},
\end{equation}
where $\bm{R}_\mathcal{T}\bu$ denotes the cell component of $\hat{\bm{R}}_h\bu$.
\end{theorem}

\begin{proof}
Set $\bm{\hat\epsilon}_h:=\Iint_h\bu-\hat{\bm{R}}_h\bu$ and take
$\bvt:=\bm\epsilon_\mathcal{T}$ in \eqref{eq:dual}, with solution $\bzeta$. Integrating by
parts cellwise, using that $\bsig(\bzeta)\bn$ is single valued across the mesh interfaces not
lying on $\Gamma$, that the face components of $\bm{\hat\epsilon}_h$ are single valued and
vanish on $\partial\Omega$, and inserting the reconstructed dual stress
$\bG_{T^i}(\Iint_T\bzeta)$ exactly as in \cite[Prop.~5.2]{HuangBurman2024}, one obtains
\[
\|\bm\epsilon_\mathcal{T}\|_\Omega^2
=\Xi_h^{\rm ell}\big(\bzeta;\bm{\hat\epsilon}_h\big)
+a_h\big(\bm{\hat\epsilon}_h,\Iint_h\bzeta\big),
\]
where in the first term the roles of the primal and the dual solutions are interchanged. The
first term is bounded by $|\bzeta|_{\sharp h}\tnorm{\bm{\hat\epsilon}_h}_h$ thanks to
\Cref{lem:consistency}, and $|\bzeta|_{\sharp h}\lesssim h\,\Phi_0(\bzeta)$ by
\Cref{lem:sharpapprox} applied with $k=0$, i.e.\ using only the $H^2$-regularity
\eqref{eq:reg}. The second term is treated as in \cite[Steps~4--6~of~Proposition~5.2]{HuangBurman2024}: 
\begin{align*}
	a_h&(\bm{\hat\epsilon}_h,\Iint_h(\bzeta))
	=\sum_{T\in\Th}\Big[ (\nabla\cdot\bsig(\bu),\bm{I}_T(\bzeta)-\bzeta)_T
	-a_T( \bu-\Iint_T \bu, \bzeta-\Iint_T \bzeta)\\
	&+\big\langle S_h \jump{\bu-\bm{I}_T\bu}, K \big(\bsig(\bzeta_1) -\bG_{T^1}(\Iint_T\bzeta)\big)\bn \big\rangle_{T^\Gamma}
	+\big\langle \bu_2 -\bm{I}_{T^2}\bu, \jump{ \big( \bsig(\bzeta)-\bG_T(\Iint_T\bzeta) \big)\bn }\big\rangle_{T^\Gamma} \\
	&+\big\langle K \big(\bsig(\bu_1) -\bG_{T^1}(\Iint_T\bn)\big)\bn, S_h\jump{\bzeta -\bm{I}_T\bzeta} \big\rangle_{T^\Gamma}
	+\big\langle \jump{ \big(\bsig(\bu) -\bG_T(\Iint_T\bu)\big)\bn }, \bzeta_2 -\bm{I}_{T^2}\bzeta\big\rangle_{T^\Gamma} \\
	&+\langle S_h \jump{\bm{I}_T\bu-\bu}, \jump{\bm{I}_T\bzeta -\bzeta} \rangle_{T^\Gamma} +s_T^{\partial T}(\Iint_T(\bu),\Iint_T(\bzeta)) \Big],
\end{align*}
using
\eqref{eq:ritz} with the test function $\Iint_h\bzeta$, subtracting
$\big(\bm{I}_T(\nabla\cdot\bsig(\bu)),\bm{I}_T(\bzeta)-\bzeta\big)_T$, and expanding the interface stabilisation by means of the
two interface conditions satisfied by $\bu$ and by $\bzeta$, all the leading contributions
cancel and the remainder is bounded by
$h^{k+2}\sum_{i=1,2}\mu_i\,(1+\kappa_i)^2|\bu|_{H^{k+2}(\Omega_i)}
\|\bm\epsilon_\mathcal{T}\|_\Omega$; the only place where the unfitted setting differs from
the fitted one is that the pairings on $\partial T\cap\Gamma$ are replaced by pairings on
$T^\Gamma$ and $(\partial T)^i$, which are handled with \Cref{lem:disc,lem:mtr} instead of
their fitted counterparts. Combining the two bounds with \eqref{eq:ritzenergy} and dividing
by $\|\bm\epsilon_\mathcal{T}\|_\Omega$ gives \eqref{eq:ritzL2}.
\end{proof}

We can now state the $L^2$-error estimate for the space semi-discrete problem. Define
\begin{equation}\label{eq:Psi}
\Psi_k(t,\bu):=|\bu|_{L^\infty(0,t;\bm{H}^{k+2}(\Omega_1\cup\Omega_2))}
+|\dt\bu|_{L^1(0,t;\bm{H}^{k+2}(\Omega_1\cup\Omega_2))}.
\end{equation}

\begin{proposition}[$L^2$-error estimate]\label{prop:L2}
Let $\bu$ solve \eqref{eq:model} and let $\bhu_h$ solve \eqref{eq:semi} with the initial
condition \eqref{eq:ic}. Assume $k\ge1$, that \eqref{eq:reg} holds and that
$\bu\in C^1(\overline J;\bm{H}^{k+2}(\Omega_1\cup\Omega_2))$. Then, for all $t\in J$,
\begin{equation}\label{eq:L2est}
\big\|\varrho^{1/2}\big(\bu_\mathcal{T}-\bm{I}_\mathcal{T}\bu\big)
\big\|_{C^0(0,t;\bm{L}^2(\Omega))}\lesssim
h^{k+2}\,
\Psi_k(t,\bu).
\end{equation}
\end{proposition}

\begin{proof}
\emph{Step 1: error equation.} We use a different error decomposition than in
\Cref{prop:energy}: we set $\hat{\bm{d}}_h(t):=\bhu_h(t)-\hat{\bm{R}}_h\bu(t)$ for $t\in J$,
where $\hat{\bm{R}}_h$ is the elliptic projection \eqref{eq:ritz}. From \eqref{eq:semi} and
\eqref{eq:ritz} we infer that, for all $\bhw_h\in\Vhz$,
\[
\big(\varrho\,\dtt\bm{d}_\mathcal{T}(t),\bw_\mathcal{T}\big)_\Omega
+a_h\big(\hat{\bm{d}}_h(t),\bhw_h\big)
=\big(\varrho\,\dtt\bm\varpi(t),\bw_\mathcal{T}\big)_\Omega
+\big(\varrho\big(\dtt\bu-\dtt\bm{I}_\mathcal{T}\bu\big),\bw_\mathcal{T}\big)_{\Omega^\Gamma},
\]
with $\bm\varpi(t):=\bm{I}_\mathcal{T}\bu(t)-\bm{R}_\mathcal{T}\bu(t)$, where we used
$-\nabla\!\cdot\!\bsig(\bu)=\bff-\varrho\dtt\bu$ and the fact that $\bm{I}_{T^i}$ coincides
with the $L^2$-projection on the uncut cells.

\emph{Step 2: stability argument.} Let $\chi\in J$ and set
$\hat{\bm{z}}_h(t):=\int_t^{\chi}\hat{\bm{d}}_h(s)\,ds$, so that
$\dt\hat{\bm{z}}_h=-\hat{\bm{d}}_h$ and $\hat{\bm{z}}_h(\chi)=\bm{0}$. Testing the above
identity with $\bhw_h=\hat{\bm{z}}_h(t)$, integrating by parts in time and using the symmetry
of $a_h$, we obtain, exactly as in \cite[Thm.~2]{BDES2021},
\begin{multline*}
\tfrac14\big\|\varrho^{1/2}\bm{d}_\mathcal{T}\big\|_{C^0(0,t;\bm{L}^2(\Omega))}^2
\le\tfrac12\big\|\varrho^{1/2}\bm\varpi(0)\big\|_\Omega^2
+\big\|\varrho^{1/2}\dt\bm\varpi\big\|_{L^1(0,t;\bm{L}^2(\Omega))}^2\\
+\big\|\varrho^{1/2}\dt\big(\bu-\bm{I}_\mathcal{T}\bu\big)
\big\|_{L^1(0,t;\bm{L}^2(\Omega^\Gamma))}^2,
\end{multline*}
where we also used that $\bm\varpi(0)-\bm{d}_\mathcal{T}(0)=\bm{0}$ and
$a_h(\hat{\bm{z}}_h(0),\hat{\bm{z}}_h(0))\ge0$.

\emph{Step 3: conclusion.} Since
$\bu_\mathcal{T}-\bm{I}_\mathcal{T}\bu=\bm{d}_\mathcal{T}-\bm\varpi$, the triangle
inequality gives
\begin{multline*}
\big\|\varrho^{1/2}(\bu_\mathcal{T}-\bm{I}_\mathcal{T}\bu)\big\|_{C^0(0,t;\bm{L}^2(\Omega))}
\lesssim\big\|\varrho^{1/2}\bm\varpi\big\|_{C^0(0,t;\bm{L}^2(\Omega))}
+\big\|\varrho^{1/2}\dt\bm\varpi\big\|_{L^1(0,t;\bm{L}^2(\Omega))}\\
+\big\|\varrho^{1/2}\dt(\bu-\bm{I}_\mathcal{T}\bu)\big\|_{L^1(0,t;\bm{L}^2(\Omega^\Gamma))}.
\end{multline*}
The first two terms are bounded by \Cref{thm:ritzL2} applied to $\bu(t)$ and to $\dt\bu(t)$
(note that $\hat{\bm{R}}_h$ commutes with $\dt$, since the data of \eqref{eq:ritz} depend
linearly on $\bu$), and the last one is of order $h^{k+2}$ by \Cref{lem:approx}, which gives
\eqref{eq:L2est} with $\Psi_k$ defined in \eqref{eq:Psi}.
\end{proof}

\begin{remark}[Optimality]\label{rem:optimal}
The rates $h^{k+1}$ in \Cref{prop:energy} and $h^{k+2}$ in \Cref{prop:L2} are optimal for the
polynomial degrees used, namely $k$ on the faces and $k+1$ in the cells. The constants do not
depend on the compliancy tensor $\bK$, nor on how $\Gamma$ cuts the mesh, and they depend on
the material parameters only through the explicit factors indicated in \eqref{eq:sharpapprox} and
\eqref{eq:ritzL2}. All the estimates degenerate as $\kappa_i\to+\infty$, see
\Cref{rem:locking}.
\end{remark}

\section{Time discretisation}\label{sec:time}

\subsection{Algebraic setting and the Newmark scheme}\label{ssec:newmark}

Let $N^{k+1}_\mathcal{T}:=\dim(\bm{V}_\mathcal{T})$ and
$N^k_\mathcal{F}:=\dim(\bm{V}_{\mathcal{F}0})$, where $\bm{V}_{\mathcal{F}0}$ is the subspace
of $\bm{V}_\mathcal{F}$ with zero components on the boundary faces. Once bases have been
chosen, the algebraic realisation of \eqref{eq:semi} reads
\begin{equation}\label{eq:algebraic}
\begin{bmatrix}\mathsf{M}_{\mathcal{T}\mathcal{T}}&0\\0&0\end{bmatrix}
\begin{bmatrix}\dtt\mathsf{U}_\mathcal{T}(t)\\ \bullet\end{bmatrix}
+\begin{bmatrix}\mathsf{K}_{\mathcal{T}\mathcal{T}}&\mathsf{K}_{\mathcal{T}\mathcal{F}}\\
\mathsf{K}_{\mathcal{F}\mathcal{T}}&\mathsf{K}_{\mathcal{F}\mathcal{F}}\end{bmatrix}
\begin{bmatrix}\mathsf{U}_\mathcal{T}(t)\\ \mathsf{U}_\mathcal{F}(t)\end{bmatrix}
=\begin{bmatrix}\mathsf{F}_\mathcal{T}(t)\\ \mathsf{F}_\mathcal{F}(t)\end{bmatrix},
\end{equation}
with the (block-diagonal) mass matrix $\mathsf{M}_{\mathcal{T}\mathcal{T}}$ associated with
the weighted inner product $(\varrho \cdot,\cdot)_{\Omega}$ and the symmetric positive-definite stiffness
matrix associated with $a_h$; the bullet stands for $\dtt\mathsf{U}_\mathcal{F}$, which is
irrelevant owing to the structure of the mass matrix.

Let $(t^n)_{0\le n\le N}$ be the discrete time nodes with $t^0=0$ and $t^N=T$, and let
$\Delta t=T/N$. We consider the Newmark scheme with parameters $\beta$ and $\gamma$, which is
second-order accurate in time, implicit if $\beta>0$ and unconditionally stable if
$\tfrac12\le\gamma\le2\beta$; the classical choice is $\gamma=\tfrac12$, $\beta=\tfrac14$.
The scheme considers an approximation for the displacement, the velocity and the
acceleration at each time node, which are all hybrid unknowns
$\bhu_h^n,\bhv_h^n,\hat{\bm{a}}_h^n\in\Vhz$. It is initialised by
$\bhu_h^0:=\Iint_h(\bu_0)$, $\bhv_h^0:=\Iint_h(\bq_0)$, and $\hat{\bm{a}}^0_h$ is defined by
$(\varrho \bm{a}^0_\mathcal{T},\bq_\mathcal{T})_{\Omega}
+a_h(\bhu^0_h,(\bq_\mathcal{T},\bm{0}))=\ell_h((\bq_\mathcal{T},\bm{0}))$ for all
$\bq_\mathcal{T}\in\bm{V}_\mathcal{T}$ and
$a_h(\hat{\bm{a}}^0_h,(\bm{0},\bq_\mathcal{F}))=\ell_h((\bm{0},\bq_\mathcal{F}))$
for all $\bq_\mathcal{F}\in\bm{V}_{\mathcal{F}0}$. Then, for all
$n\in\{0,\dots,N-1\}$, one performs the following three steps.
\begin{enumerate}[leftmargin=1.4em,itemsep=1pt]
\item \emph{Predictor:}
$\bhu_h^{*n}:=\bhu_h^n+\Delta t\,\bhv_h^n+\tfrac12\Delta t^2(1-2\beta)\hat{\bm{a}}_h^n$ and
$\bhv_h^{*n}:=\bhv_h^n+\Delta t(1-\gamma)\hat{\bm{a}}_h^n$.
\item \emph{Linear solver:} find $\hat{\bm{a}}_h^{n+1}\in\Vhz$ such that, for all
$\bhq_h\in\Vhz$,
\begin{equation}\label{eq:newmarksolve}
\big(\varrho \bm{a}^{n+1}_\mathcal{T},\bq_\mathcal{T}\big)_{\Omega}
+\beta\Delta t^2\,a_h\big(\hat{\bm{a}}_h^{n+1},\bhq_h\big)
=\ell_h^{n+1}(\bhq_h)-a_h\big(\bhu_h^{*n},\bhq_h\big).
\end{equation}
\item \emph{Corrector:} $\bhu_h^{n+1}:=\bhu_h^{*n}+\beta\Delta t^2\hat{\bm{a}}_h^{n+1}$ and
$\bhv_h^{n+1}:=\bhv_h^{*n}+\gamma\Delta t\,\hat{\bm{a}}_h^{n+1}$.
\end{enumerate}
Since the matrix
$\mathsf{M}_{\mathcal{T}\mathcal{T}}+\beta\Delta t^2\mathsf{K}_{\mathcal{T}\mathcal{T}}$
is block diagonal, static condensation can be applied to \eqref{eq:newmarksolve}: the cell
unknowns are eliminated locally and only a global transmission problem coupling the face
unknowns attached to $\bm{V}_{\mathcal{F}0}$ has to be solved, with a matrix that is
independent of $n$ and can therefore be factorised once and for all.

An important property of the scheme is its energy balance. Let
\begin{equation}\label{eq:discenergy}
\hat{E}^n:=\tfrac12\big\|\varrho^{1/2}\bm{v}^n_\mathcal{T}\big\|_\Omega^2
+\tfrac12a_h\big(\bhu^n_h,\bhu^n_h\big)
+\varsigma\Delta t^2\big\|\varrho^{1/2}\bm{a}^n_\mathcal{T}\big\|_\Omega^2,
\qquad\varsigma:=\tfrac14(2\beta-\gamma).
\end{equation}
A straightforward extension of \cite[Lem.~4]{BurmanDuranErn2022CAMC} (the discrete bilinear
form $a_h$ being symmetric positive semi-definite, which is all that is used there) shows
that
\begin{equation}\label{eq:energybalance}
\hat E^n=\hat E^0+\sum_{m=0}^{n-1}\tfrac12\big(\bff(t^{m+1})+\bff(t^m),
\bm{u}^{m+1}_\mathcal{T}-\bm{u}^m_\mathcal{T}\big)_\Omega
+\sum_{m=0}^{n-1}\tfrac12\big(\ell^{\Gamma,m+1}_h+\ell^{\Gamma,m}_h\big)
\big(\bhu^{m+1}_h-\bhu^m_h\big),
\end{equation}
where $\ell_h^{\Gamma,m}$ collects the interface contributions to $\ell_h$ at time $t^m$. In
particular, for $\beta=\tfrac14$, $\gamma=\tfrac12$ (so that $\varsigma=0$) and in the
absence of external forcing and of interface data, the discrete energy
$\hat E^n=\tfrac12\|\varrho^{1/2}\bm{v}^n_\mathcal{T}\|^2_\Omega
+\tfrac12a_h(\bhu^n_h,\bhu^n_h)$ is \emph{exactly} conserved, uniformly in the mesh size, in
the polynomial degree and in the compliancy. This is confirmed numerically in
\Cref{ssec:case2}.

\subsection{Singly diagonally implicit Runge--Kutta schemes}\label{ssec:sdirk}

The Newmark scheme is only second-order accurate in time and, for
$\beta=\tfrac14$, $\gamma=\tfrac12$, it is non-dissipative. In wave propagation over many
periods, its dispersion error accumulates and, unless $\Delta t$ is taken very small, it
soon dominates the space discretisation error; moreover, nothing damps the poorly resolved
high-frequency components of the discrete solution. Higher-order, $A$-stable and mildly
dissipative schemes are therefore attractive, and we consider here the singly diagonally
implicit Runge--Kutta (SDIRK) schemes used for the wave equation in
\cite[Sect.~4.2]{BurmanDuranErn2022CAMC} and \cite[Sect.~5.2.3]{BDES2021}.

To this purpose, we rewrite the space semi-discrete problem \eqref{eq:semi} as a first-order
system in time by introducing the hybrid velocity $\bhv_h:=\dt\bhu_h$, which gives
\begin{equation}\label{eq:firstorder}
\dt\bhu_h=\bhv_h,\qquad
\big(\varrho\,\dt\bm{v}_\mathcal{T},\bw_\mathcal{T}\big)_\Omega+a_h(\bhu_h,\bhw_h)
=\ell_h(\bhw_h)\quad\forall\,\bhw_h\in\Vhz .
\end{equation}
We emphasise that only the \emph{time} discretisation is modified: the space discretisation
of \Cref{sec:method}, and hence the whole analysis of \Cref{sec:analysis}, is unchanged. This
differs from \cite{BurmanDuranErn2022CAMC,BDES2021}, where the Runge--Kutta schemes are
applied to a genuinely first-order HHO formulation in which the stress tensor is an
independent cellwise unknown.

Recall that a Runge--Kutta scheme is defined by its Butcher tableau
$\{a_{ij}\}_{1\le i,j\le s}$, $\{b_i\}_{1\le i\le s}$ and $\{c_i\}_{1\le i\le s}$,
where $s\ge1$ is the number of stages, and that it is SDIRK when the matrix
$\{a_{ij}\}$ is lower triangular with $a_{11}=\dots=a_{ss}=:a_*$. We consider the
SDIRK$(s,s+1)$ schemes with $s$ stages and order $s+1$, $s\in\{1,2,3\}$, whose Butcher
tableaux are \cite[Sect.~5]{BurmanDuranErn2022CAMC}
\begin{equation}\label{eq:tableaux}
\renewcommand{\arraystretch}{1.2}
\begin{array}{c|c}
\tfrac12&\tfrac12\\\hline &1
\end{array}
\qquad
\begin{array}{c|cc}
\gamma&\gamma&0\\
1-\gamma&1-2\gamma&\gamma\\\hline
&\tfrac12&\tfrac12
\end{array}
\qquad
\begin{array}{c|ccc}
\gamma&\gamma&0&0\\
\tfrac12&\tfrac12-\gamma&\gamma&0\\
1-\gamma&2\gamma&1-4\gamma&\gamma\\\hline
&\delta&1-2\delta&\delta
\end{array}
\end{equation}
with $\gamma:=\tfrac12+\tfrac{1}{2\sqrt3}$ for $s=2$, and
$\gamma:=\tfrac{1}{\sqrt3}\cos\big(\tfrac{\pi}{18}\big)+\tfrac12$,
$\delta:=\tfrac{1}{6(2\gamma-1)^2}$ for $s=3$.

Let $\bhu^0_h:=\Iint_h(\bu_0)$ and $\bhv^0_h:=\Iint_h(\bq_0)$, the face components being
determined as in \Cref{ssec:semidiscrete}; no initial acceleration is needed. For all
$n\in\{1,\dots,N\}$, the step from $t^{n-1}$ to $t^n$ proceeds as follows. For all
$i\in\{1,\dots,s\}$, set
\begin{equation}\label{eq:rkpred}
\bhu^{*n,i}_h:=\bhu^{n-1}_h+\Delta t\sum_{j<i}a_{ij}\bhv^{n,j}_h,\qquad
\bhv^{*n,i}_h:=\bhv^{n-1}_h+\Delta t\sum_{j<i}a_{ij}\hat{\bm{\kappa}}^{n,j}_h,
\end{equation}
and $\bhp^{n,i}_h:=\bhu^{*n,i}_h+a_*\Delta t\,\bhv^{*n,i}_h$; then find
$\hat{\bm{\kappa}}^{n,i}_h\in\Vhz$ such that, for all $\bhq_h\in\Vhz$,
\begin{equation}\label{eq:rksolve}
\big(\varrho \bm{\kappa}^{n,i}_\mathcal{T},\bq_\mathcal{T}\big)_{\Omega}
+(a_*\Delta t)^2\,a_h\big(\hat{\bm{\kappa}}^{n,i}_h,\bhq_h\big)
=\ell_h^{n-1+c_i}(\bhq_h)-a_h\big(\bhp^{n,i}_h,\bhq_h\big),
\end{equation}
and set $\bhv^{n,i}_h:=\bhv^{*n,i}_h+a_*\Delta t\,\hat{\bm{\kappa}}^{n,i}_h$. Finally,
\begin{equation}\label{eq:rkupdate}
\bhu^n_h:=\bhu^{n-1}_h+\Delta t\sum_{i=1}^sb_i\bhv^{n,i}_h,\qquad
\bhv^n_h:=\bhv^{n-1}_h+\Delta t\sum_{i=1}^sb_i\hat{\bm{\kappa}}^{n,i}_h .
\end{equation}

Three comments are in order. First, comparing \eqref{eq:rksolve} with \eqref{eq:newmarksolve}
shows that a Runge--Kutta stage has exactly the same structure as a Newmark step, with
$\beta\Delta t^2$ replaced by $(a_*\Delta t)^2$; in particular, static condensation applies
verbatim. Second, since the scheme is \emph{singly} diagonally implicit, the matrix of
\eqref{eq:rksolve} is the same at every stage and at every time step, so that a single
factorisation of the condensed matrix is performed once and for all, exactly as for Newmark;
the cost of one time step is then $s$ times that of a Newmark step. Third, taking
$\bhq_h=(\bm{0},\bq_\mathcal{F})$ in \eqref{eq:rksolve} and using
$\bhu^{n,i}_h=\bhp^{n,i}_h+(a_*\Delta t)^2\hat{\bm\kappa}^{n,i}_h$ shows that
$a_h(\bhu^{n,i}_h,(\bm{0},\bq_\mathcal{F}))=\ell^{n-1+c_i}_h((\bm{0},\bq_\mathcal{F}))$, i.e.\
the algebraic equations attached to the faces --- the second block row of
\eqref{eq:algebraic} --- are satisfied exactly at every stage. The same equations are
re-imposed on $\bhu^n_h$ and $\bhv^n_h$ after the update \eqref{eq:rkupdate}, which only
affects their face components.

\begin{remark}[Energy]\label{rem:rkenergy}
The SDIRK schemes \eqref{eq:tableaux} are $A$-stable but, contrary to the Newmark scheme with
$\beta=\tfrac14$, $\gamma=\tfrac12$, they are dissipative: the discrete energy
$\hat E^n$ of \eqref{eq:discenergy} (with $\varsigma:=0$) decays in the absence of external
forcing. The dissipation is concentrated on the poorly resolved high-frequency components,
which is precisely the mechanism that makes these schemes more accurate than the
conservative one at a given $\Delta t$; the energy loss is quantified numerically in
\Cref{ssec:case2}.
\end{remark}

\section{Numerical experiments}\label{sec:num}

In this section we report two-dimensional numerical experiments that assess the error
estimates of \Cref{prop:energy,prop:L2}, i.e.\ the convergence rates $h^{k+1}$ and $h^{k+2}$
of the energy- and $L^2$-errors, their robustness with respect to the compliancy parameters
$\alpha,\beta\ge0$ and with respect to how the interface cuts the mesh, and that illustrate
the propagation of elastic waves across an unresolved imperfect interface. Three test cases
are considered, in increasing order of difficulty: a stationary problem with a manufactured
solution (\Cref{ssec:case1}), which is then made time-dependent (\Cref{ssec:case1dyn}); the
propagation of a Ricker wave in a heterogeneous medium separated by a perfectly bonded but
unresolved interface (\Cref{ssec:case2}); and the same configuration with a linear slip
interface (\Cref{ssec:case3}).

\subsection{Implementation}\label{ssec:impl}

All the computations reported below were performed with a plain \texttt{Python/NumPy}
implementation written for verification purposes; since no attempt was made at optimising it,
we do not report computational times.

The background meshes are uniform Cartesian meshes of the computational domain, and the
interface is a straight line $\Gamma=\{\bx\mid\bn\cdot\bx=c\}$ which is not aligned with the
mesh. Choosing a straight interface allows us to compute the geometry of the cut subcells
\emph{exactly}, by clipping the (convex) cells with the half-planes defined by $\Gamma$, so
that no geometric error pollutes the convergence rates; this isolates the effect of the
unfitted discretisation itself. Integrals over the subcells $T^i$ are evaluated by a fan
triangulation from the barycentre followed by a collapsed Gauss rule that is exact for
polynomials of degree $2k+4$, and integrals over the subfaces $F^i$ and over $T^\Gamma$ by
Gauss--Legendre rules of the same accuracy.

As is customary for hybrid high-order methods \cite{CicuttinDiPietroErn2018}, an important
step in the implementation is the choice of a basis for each of the polynomial spaces
appearing in the construction. For all $T\in\Th$, $i\in\{1,2\}$ and $l\in\{k,k+1\}$ we use the
hierarchical basis of $\Pk{l}(T^i)$ obtained by $L^2(T^i)$-orthonormalisation (Cholesky
factorisation of the mass matrix) of the monomials in the scaled variables
$\bm\xi_{T^i}=(\bx-\bx_{T^i})/h_T$, where $\bx_{T^i}$ is the barycentre of the subcell $T^i$;
a basis of $\bPk{l}(T^i)$, resp.\ of $\sPk{l}(T^i)$, is then obtained by taking the Cartesian
product with the canonical basis of $\mathbb{R}^d$, resp.\ with an orthonormal basis of the
symmetric $d\times d$ matrices. The subfaces are handled in the same way, using scaled
Legendre polynomials in the arclength, normalised by $|F^i|^{-1/2}$. Since the bases are
$L^2$-orthonormal on the subcells and subfaces, the mass matrices in \eqref{eq:E1},
\eqref{eq:E2} and in the local $L^2$-projections are identity matrices, and the scaling of
the face basis functions by $|F^i|^{-1/2}$ tames the potential ill-conditioning due to
small-cut faces, cf.\ \cite[Sect.~5.3]{BCDE2021}. For every cut cell, the reconstruction
\eqref{eq:E1} amounts to solving a linear system of size $\tfrac12d(d+1)\dim\Pk{k}(T^1)$
whose matrix is the identity plus the contribution of the interface term. The cell unknowns
are eliminated by static condensation, so that the global linear system is positive definite and
only involves the face unknowns; in the time-dependent case the matrix of the condensed
system is independent of the time step and is factorised once.

Concerning the agglomeration, we use a simplified two-stage version of
\cite[Alg.~1]{BCDE2021}: a cut cell $T$ is declared to have a small cut whenever
$\min_i|T^i|\le\alpha_{\rm agg}|T|$ with $\alpha_{\rm agg}=0.3$; such a cell is then merged
with the face neighbour $T'$ that has no small cut on the same side of $\Gamma$ and whose
area on the other side is the smallest, in the spirit of \cite[Rem.~4.4]{BCDE2021}. Each cell
may take part in at most one merge, so that the agglomeration never propagates and the
agglomerated cells are unions of two neighbouring cells. Since the faces are left unchanged
except for the one that is withdrawn, the agglomerated cells are polygons with more than four
faces, which the HHO framework of \Cref{ssec:meshes} accommodates without any modification.

The time discretisation is either the Newmark scheme of \Cref{ssec:newmark} with
$\beta=\tfrac14$ and $\gamma=\tfrac12$, or one of the SDIRK$(s,s+1)$ schemes of
\Cref{ssec:sdirk}. In the latter case the condensed matrix of \eqref{eq:rksolve} is
factorised once, and the cost of a time step is $s$ times that of a Newmark step. Nonhomogeneous Dirichlet conditions, when present, are enforced strongly by
prescribing the boundary face unknowns to be the $L^2$-projections of the exact displacement.
Setting $\bhe_h=\bhu_h-\Iint_h\bu$, we monitor the two error measures
\begin{equation}\label{eq:errmeasures}
E_e:=a_h(\bhe_h,\bhe_h)^{1/2},\qquad E_0:=\|\bm{e}_\mathcal{T}\|_\Omega,
\end{equation}
which by \Cref{thm:stab} are equivalent to $\tnorm{\bhe_h}_h$ and to the $L^2$-error of
\Cref{prop:L2}; $\tilde E_e$ and $\tilde E_0$ denote the corresponding relative errors,
normalised by $a_h(\Iint_h\bu,\Iint_h\bu)^{1/2}$ and by
$\|\bm{I}_\mathcal{T}\bu\|_\Omega$. In the time-dependent case we report
$E_e^\infty:=\max_n E_e(t^n)$ and $E_0^\infty:=\max_n E_0(t^n)$, in agreement with the
$C^0(0,t;\cdot)$ norms of \Cref{prop:energy,prop:L2}. As an elementary validation of the
implementation we checked that, for $k\in\{1,2,3\}$ and for any value of $\alpha$ and
$\beta$, every rigid-body motion $\bm{r}\in RM$ satisfies
$a_h(\Iint_h\bm{r},\bhw_h)=0$ for all $\bhw_h\in\Vh$ to machine precision, which exercises
the reconstruction, the stabilisation and the interface terms simultaneously.

\subsection{Test case 1: stationary problem with an unfitted slip interface}\label{ssec:case1}

We let $\Omega=(0,1)^2$ and $\Gamma=\{\bx\mid\bn\cdot\bx=c\}$ with
$\bn=\tfrac{1}{\sqrt5}(2,1)^{\mathsf T}$ and $c=\bn\cdot(\tfrac12,\tfrac12)^{\mathsf T}$, so
that $\Gamma$ is a straight segment crossing the square obliquely and is resolved by no
Cartesian mesh. We take $E_1=1$, $\nu_1=0.25$ in $\Omega_1=\{\bn\cdot\bx<c\}$ and $E_2=10$,
$\nu_2=0.3$ in $\Omega_2$, so that $\mu_1=0.4<\mu_2\simeq3.85$ as required by
\eqref{eq:softside}, and unless stated otherwise $\alpha=\beta=5\cdot10^{-2}$.

We manufacture an exact solution satisfying the interface conditions of \eqref{eq:model} with
$\bg_D=\bm{0}$ and $\bg_N=\bm{0}$ \emph{exactly for every} $\alpha\ge0$, $\beta\ge0$, by a
Taylor expansion in the normal direction, which is the analogue for an oblique interface of
the construction used in the fitted case \cite{HuangBurman2024}. Let
$\bt=(-n_2,n_1)^{\mathsf T}$, $s(\bx)=\bn\cdot\bx-c$ and $\tau(\bx)=\bt\cdot\bx$, and let
\begin{equation}\label{eq:u2}
\bu_2:=\mathbf{curl}\,\psi+\frac{\bx}{2\lambda_2},\qquad
\psi(\bx)=\sin\big(\pi x_1+\tfrac\pi4\big)\sin(2\pi x_2),
\end{equation}
so that $\nabla\!\cdot\!\bu_2=\lambda_2^{-1}$, which keeps $\lambda|\nabla\!\cdot\!\bu|$
bounded as $\nu\to\tfrac12$. On $\Gamma$, parametrised by $\tau$, set
$\bm{T}:=\bsig_2(\bu_2)\bn$, $\bm{R}:=\bm{T}-\bsig_1(\bu_2)\bn$ and
\begin{equation}\label{eq:VW}
\bm{V}:=-\bK\bm{T},\qquad
W_n:=\frac{\bm{R}\cdot\bn-\lambda_1\partial_\tau V_t}{2\mu_1+\lambda_1},\qquad
W_t:=\frac{\bm{R}\cdot\bt}{\mu_1}-\partial_\tau V_n,\qquad \bm{W}:=W_n\bn+W_t\bt,
\end{equation}
where $V_n=\bm{V}\cdot\bn$ and $V_t=\bm{V}\cdot\bt$, and finally
\begin{equation}\label{eq:u1}
\bu_1:=\bu_2+\bm{V}(\tau)+s\,\bm{W}(\tau)
-\tfrac12s^2\,\partial_\tau W_t(\tau)\,\bn,\qquad \bff_i:=-\nabla\!\cdot\!\bsig_i(\bu_i).
\end{equation}
By construction $\bsig_1(\bu_1)\bn=\bsig_2(\bu_2)\bn$ and
$\jump{\bu}=\bm{V}=-\bK\bsig(\bu_1)\bn$ on $\Gamma$; the quadratic term in \eqref{eq:u1} is
included so that $\nabla\!\cdot\!\bu_1=\mathcal{O}(\lambda_1^{-1})$ irrespective of $\alpha$
and $\beta$. The fields $\bu_i$ and $\bff_i$ are computed by symbolic differentiation, and we
verified that the residuals of the two interface conditions are at the level of the machine
precision for every value of the compliancy used below.

\Cref{fig:mesh} displays two meshes of the sequence together with the agglomerated cells. On
this configuration the interface cuts most cells in a rather unbalanced way, so that
essentially every cut cell is agglomerated with one of its neighbours. The errors
\eqref{eq:errmeasures} are reported in \Cref{tab:conv1} and displayed in
\Cref{fig:conv1}: the rates $k+1$ for $E_e$ and $k+2$ for $E_0$ predicted by
\Cref{prop:energy,prop:L2} are observed for $k\in\{1,2,3\}$.

\begin{figure}[htbp]\centering
\includegraphics[width=\linewidth]{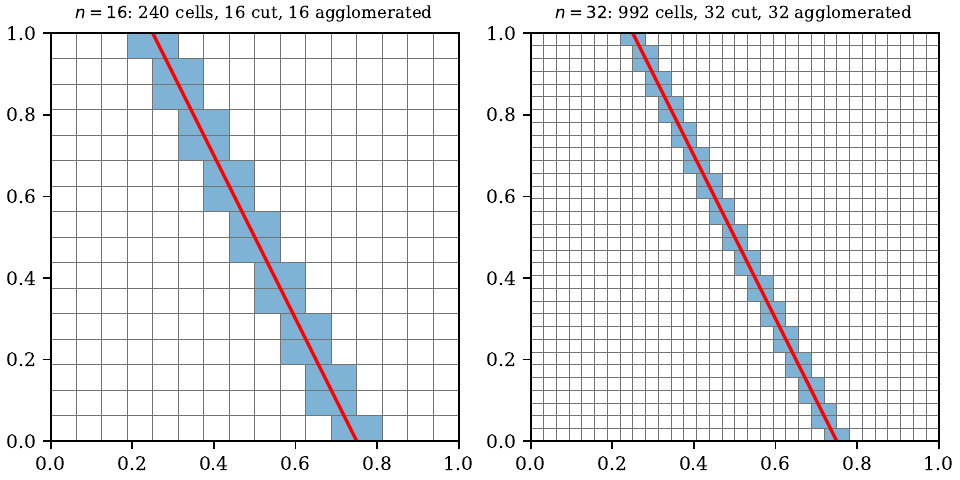}
\caption{Test case 1. Background Cartesian meshes, interface $\Gamma$ (red) and agglomerated
cells (shaded).}\label{fig:mesh}
\end{figure}

\begin{table}[htbp]\centering
\caption{Test case 1 (stationary, $\alpha=\beta=5\cdot10^{-2}$). Errors
\eqref{eq:errmeasures} and observed convergence rates.}\label{tab:conv1}
\small\setlength{\tabcolsep}{4pt}
\begin{tabular}{rrrrrr@{\quad}rr@{\quad}rr}
\toprule
$k$ & $n$ & card$(\mathcal{T}_h)$ & cut & aggl. & $N_{\rm dof}$ & $E_e$ & rate & $E_0$ & rate\\
\midrule
1 & 4 & 12 & 4 & 4 & 280 & $7.51\cdot10^{1}$ & -- & $5.46\cdot10^{0}$ & --\\
1 & 8 & 56 & 8 & 8 & 1200 & $2.32\cdot10^{1}$ & 1.70 & $7.99\cdot10^{-1}$ & 2.77\\
1 & 16 & 240 & 16 & 16 & 4960 & $6.33\cdot10^{0}$ & 1.87 & $1.02\cdot10^{-1}$ & 2.96\\
1 & 32 & 992 & 32 & 32 & 20160 & $1.65\cdot10^{0}$ & 1.94 & $1.29\cdot10^{-2}$ & 2.99\\
\midrule
2 & 4 & 12 & 4 & 4 & 452 & $1.83\cdot10^{1}$ & -- & $1.22\cdot10^{0}$ & --\\
2 & 8 & 56 & 8 & 8 & 1928 & $2.49\cdot10^{0}$ & 2.88 & $7.45\cdot10^{-2}$ & 4.04\\
2 & 16 & 240 & 16 & 16 & 7952 & $3.20\cdot10^{-1}$ & 2.96 & $4.31\cdot10^{-3}$ & 4.11\\
2 & 32 & 992 & 32 & 32 & 32288 & $4.06\cdot10^{-2}$ & 2.98 & $2.59\cdot10^{-4}$ & 4.06\\
\midrule
3 & 4 & 12 & 4 & 4 & 656 & $3.18\cdot10^{0}$ & -- & $1.80\cdot10^{-1}$ & --\\
3 & 8 & 56 & 8 & 8 & 2784 & $2.18\cdot10^{-1}$ & 3.87 & $6.31\cdot10^{-3}$ & 4.83\\
3 & 16 & 240 & 16 & 16 & 11456 & $1.48\cdot10^{-2}$ & 3.88 & $2.27\cdot10^{-4}$ & 4.79\\
3 & 32 & 992 & 32 & 32 & 46464 & $9.66\cdot10^{-4}$ & 3.94 & $7.64\cdot10^{-6}$ & 4.89\\
\bottomrule
\end{tabular}
\end{table}

\begin{figure}[htbp]\centering
\includegraphics[width=\linewidth]{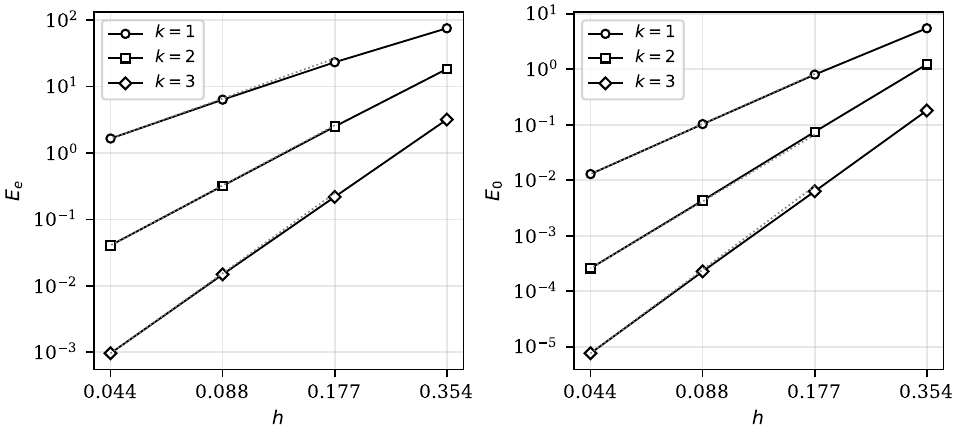}
\caption{Test case 1. Energy error $E_e$ (left) and $L^2$-error $E_0$ (right) versus $h$; the
dotted lines have slopes $k+1$ and $k+2$.}\label{fig:conv1}
\end{figure}

\subsection{Robustness with respect to the compliancy}\label{ssec:robust}

Since the exact solution itself depends on $\alpha$ and $\beta$, we report in
\Cref{tab:robust} the \emph{relative} errors $\tilde E_e$ and $\tilde E_0$ on the mesh with
$n=16$, for $\alpha=\beta$ ranging over sixteen orders of magnitude and for the degenerate
combinations in which one of the two parameters vanishes. Recall from \Cref{rem:regimes} that
$\bn^{\mathsf T}\bS_h\bn=(\beta+h_T/\delta)^{-1}$ and
$\bt^{\mathsf T}\bS_h\bt=(\alpha+h_T/\delta)^{-1}$; since $h_T/\delta$ ranges from
$4\cdot10^{-2}$ to $7\cdot10^{-2}$ on the cells adjacent to $\Gamma$ in that mesh, the whole
range from the Nitsche-like penalty regime $\alpha,\beta\ll h_T/\delta$ to the
compliancy-dominated regime $\alpha,\beta\gg h_T/\delta$ is covered. Both the errors and the
observed rates are essentially independent of $(\alpha,\beta)$ --- the relative errors vary
by at most $50\%$ over the whole table --- as predicted by the analysis, in which the
constants do not depend on the compliancy. The rows $\alpha=\beta=10^4$ and
$\alpha=\beta=10^{12}$ are identical to five significant digits: the two subdomains are then
numerically decoupled, which is the correct free-interface limit of \Cref{rem:regimes}.

\begin{table}[htbp]\centering
\caption{Test case 1. Robustness with respect to the compliancy parameters: relative errors
on the mesh with $n=16$.}\label{tab:robust}
\small\setlength{\tabcolsep}{4pt}
\begin{tabular}{rr@{\qquad}rr@{\qquad}rr@{\qquad}rr}
\toprule
& & \multicolumn{2}{c}{$k=1$} & \multicolumn{2}{c}{$k=2$} & \multicolumn{2}{c}{$k=3$}\\
\cmidrule(lr){3-4}\cmidrule(lr){5-6}\cmidrule(lr){7-8}
$\alpha$ & $\beta$ & $\tilde E_e$ & $\tilde E_0$ & $\tilde E_e$ & $\tilde E_0$ & $\tilde E_e$ & $\tilde E_0$\\
\midrule
$0$ & $0$ & $1.48\cdot10^{-2}$ & $1.62\cdot10^{-3}$ & $7.42\cdot10^{-4}$ & $6.63\cdot10^{-5}$ & $3.42\cdot10^{-5}$ & $3.49\cdot10^{-6}$\\
$10^{-4}$ & $10^{-4}$ & $1.48\cdot10^{-2}$ & $1.62\cdot10^{-3}$ & $7.42\cdot10^{-4}$ & $6.63\cdot10^{-5}$ & $3.42\cdot10^{-5}$ & $3.49\cdot10^{-6}$\\
$10^{-2}$ & $10^{-2}$ & $1.49\cdot10^{-2}$ & $1.62\cdot10^{-3}$ & $7.44\cdot10^{-4}$ & $6.67\cdot10^{-5}$ & $3.43\cdot10^{-5}$ & $3.51\cdot10^{-6}$\\
$10^{0}$ & $10^{0}$ & $1.75\cdot10^{-2}$ & $1.92\cdot10^{-3}$ & $9.25\cdot10^{-4}$ & $9.04\cdot10^{-5}$ & $4.50\cdot10^{-5}$ & $4.37\cdot10^{-6}$\\
$10^{4}$ & $10^{4}$ & $2.05\cdot10^{-2}$ & $2.28\cdot10^{-3}$ & $1.07\cdot10^{-3}$ & $1.10\cdot10^{-4}$ & $5.17\cdot10^{-5}$ & $4.93\cdot10^{-6}$\\
$10^{12}$ & $10^{12}$ & $2.05\cdot10^{-2}$ & $2.28\cdot10^{-3}$ & $1.07\cdot10^{-3}$ & $1.10\cdot10^{-4}$ & $5.17\cdot10^{-5}$ & $4.93\cdot10^{-6}$\\
\midrule
$0$ & $10^{-2}$ & $1.49\cdot10^{-2}$ & $1.62\cdot10^{-3}$ & $7.45\cdot10^{-4}$ & $6.65\cdot10^{-5}$ & $3.43\cdot10^{-5}$ & $3.50\cdot10^{-6}$\\
$10^{-2}$ & $0$ & $1.48\cdot10^{-2}$ & $1.61\cdot10^{-3}$ & $7.42\cdot10^{-4}$ & $6.66\cdot10^{-5}$ & $3.42\cdot10^{-5}$ & $3.50\cdot10^{-6}$\\
$0$ & $10^{0}$ & $1.71\cdot10^{-2}$ & $1.16\cdot10^{-3}$ & $9.46\cdot10^{-4}$ & $5.05\cdot10^{-5}$ & $4.63\cdot10^{-5}$ & $2.80\cdot10^{-6}$\\
$10^{0}$ & $0$ & $1.51\cdot10^{-2}$ & $1.77\cdot10^{-3}$ & $7.61\cdot10^{-4}$ & $7.99\cdot10^{-5}$ & $3.59\cdot10^{-5}$ & $3.83\cdot10^{-6}$\\
\bottomrule
\end{tabular}
\end{table}

\subsection{Small cuts and cell agglomeration}\label{ssec:smallcut}

To assess the role of the agglomeration procedure, we consider a configuration in which the
size of the small cut is controlled explicitly: on the mesh of $\Omega=(0,1)^2$ with $n=16$
we take $\bn=(0,1)^{\mathsf T}$ and $c=\tfrac12+\xi h$ with $\xi\in(0,\tfrac12]$, so that
\emph{every} cell of one row is cut, the subcell in $\Omega_2$ having area $\xi h^2$; the
manufactured solution of \Cref{ssec:case1} is used with the same data. \Cref{tab:smallcut}
reports the relative errors and the spectral condition number of the diagonally scaled
stiffness matrix as $\xi\to0$, with and without agglomeration. Without agglomeration the
condition number grows like $\xi^{-1}$, reaching $2\cdot10^5$ ($k=1$) and $7\cdot10^5$
($k=2$) for $\xi=10^{-4}$, whereas with agglomeration it is uniform in $\xi$, which is the
expected behaviour and the reason for \Cref{ass:cut}. The errors themselves are almost
unaffected over this range of $\xi$, because the linear systems are solved by a sparse direct
method; the ill-conditioning is nevertheless fatal for iterative solvers and, for smaller
$\xi$ or higher polynomial degrees, also for the accuracy. We emphasise that no
agglomeration of the \emph{faces} is performed: the scaling of the face basis functions by
$|F^i|^{-1/2}$ is sufficient to tame the small-cut faces, in agreement with the observations
of \cite[Sect.~5.3]{BCDE2021}.

\begin{table}[htbp]\centering
\caption{Test case 1 with a controlled small cut of relative size $\xi$ ($n=16$,
$\alpha=\beta=5\cdot10^{-2}$): relative errors and spectral condition number of the
diagonally scaled stiffness matrix.}\label{tab:smallcut}
\small\setlength{\tabcolsep}{4pt}
\begin{tabular}{r@{\qquad}rrr@{\qquad}rrr}
\toprule
& \multicolumn{3}{c}{with agglomeration} & \multicolumn{3}{c}{without agglomeration}\\
\cmidrule(lr){2-4}\cmidrule(lr){5-7}
$\xi$ & $\tilde E_e$ & $\tilde E_0$ & $\mathrm{cond}$ & $\tilde E_e$ & $\tilde E_0$ & $\mathrm{cond}$\\
\midrule
\multicolumn{7}{l}{\itshape $k=1$}\\
$5\cdot10^{-1}$ & $4.77\cdot10^{-3}$ & $1.63\cdot10^{-4}$ & $2.2\cdot10^{3}$ & $4.77\cdot10^{-3}$ & $1.63\cdot10^{-4}$ & $2.2\cdot10^{3}$\\
$10^{-1}$ & $6.10\cdot10^{-3}$ & $2.30\cdot10^{-4}$ & $2.2\cdot10^{3}$ & $5.95\cdot10^{-3}$ & $2.07\cdot10^{-4}$ & $2.0\cdot10^{3}$\\
$10^{-2}$ & $6.47\cdot10^{-3}$ & $2.37\cdot10^{-4}$ & $2.0\cdot10^{3}$ & $6.31\cdot10^{-3}$ & $2.20\cdot10^{-4}$ & $3.3\cdot10^{3}$\\
$10^{-3}$ & $6.51\cdot10^{-3}$ & $2.38\cdot10^{-4}$ & $2.0\cdot10^{3}$ & $6.34\cdot10^{-3}$ & $2.21\cdot10^{-4}$ & $2.2\cdot10^{4}$\\
$10^{-4}$ & $6.51\cdot10^{-3}$ & $2.38\cdot10^{-4}$ & $2.0\cdot10^{3}$ & $6.35\cdot10^{-3}$ & $2.21\cdot10^{-4}$ & $2.2\cdot10^{5}$\\
\midrule
\multicolumn{7}{l}{\itshape $k=2$}\\
$5\cdot10^{-1}$ & $2.20\cdot10^{-4}$ & $5.59\cdot10^{-6}$ & $3.4\cdot10^{3}$ & $2.20\cdot10^{-4}$ & $5.59\cdot10^{-6}$ & $3.4\cdot10^{3}$\\
$10^{-1}$ & $2.84\cdot10^{-4}$ & $8.16\cdot10^{-6}$ & $3.0\cdot10^{3}$ & $2.76\cdot10^{-4}$ & $7.14\cdot10^{-6}$ & $3.5\cdot10^{3}$\\
$10^{-2}$ & $2.99\cdot10^{-4}$ & $8.25\cdot10^{-6}$ & $3.0\cdot10^{3}$ & $2.90\cdot10^{-4}$ & $7.53\cdot10^{-6}$ & $8.6\cdot10^{3}$\\
$10^{-3}$ & $3.01\cdot10^{-4}$ & $8.26\cdot10^{-6}$ & $3.0\cdot10^{3}$ & $2.91\cdot10^{-4}$ & $7.58\cdot10^{-6}$ & $7.1\cdot10^{4}$\\
$10^{-4}$ & $3.01\cdot10^{-4}$ & $8.26\cdot10^{-6}$ & $3.0\cdot10^{3}$ & $2.91\cdot10^{-4}$ & $7.58\cdot10^{-6}$ & $7.0\cdot10^{5}$\\
\bottomrule
\end{tabular}
\end{table}

\subsection{The quasi-incompressible limit}\label{ssec:incomp}

\Cref{tab:incomp} reports the behaviour of the method in the quasi-incompressible limit,
obtained by letting $\nu_1=\nu_2=\nu\to\tfrac12$ at fixed $E_1,E_2$, so that
$\lambda_1=\lambda_2/10\to+\infty$ while $\lambda\|\nabla\!\cdot\!\bu\|_\infty$ stays
bounded. With the stabilisation \eqref{eq:stab}, the errors grow like $\lambda^{1/2}$ while
the observed rates remain optimal, which is precisely the behaviour predicted by the factor
$\Phi_k$ of \eqref{eq:sharpapprox} and analysed in \Cref{rem:locking}. The last four columns
report the same quantities computed with the anisotropic stabilisation weight
$h_T^{-1}(2\mu_i\bI+d\lambda_i\bn_T\otimes\bn_T)$ borrowed from the fitted method
\cite{HuangBurman2024}: the $L^2$-error is then uniform in $\lambda$ over four orders of
magnitude, whereas $E_e=a_h(\bhe_h,\bhe_h)^{1/2}$ still grows like $\lambda^{1/2}$, exactly
as in the fitted case \cite[Rem.~6.1]{HuangBurman2024}, because the error is measured against
the $L^2$-projection based interpolate $\Iint_h$, for which
$\|(\bu-\bm{I}_T\bu)\cdot\bn\|_{\partial (T^i)}$ is only of order $h^{k+3/2}$. A proof of the
$\lambda$-uniformity of the $L^2$-error with the anisotropic weight would require an
$H(\dive)$-conforming interpolate on the cut subcells, which is not available; we leave this
question to future work and we use the simpler stabilisation \eqref{eq:stab} in all the other
experiments.

\begin{table}[htbp]\centering
\caption{Test case 1. Quasi-incompressible limit ($\alpha=\beta=5\cdot10^{-2}$,
$\lambda_1=\lambda_2/10$): errors on the mesh with $n=16$ and rates observed on the last
refinement, for the two stabilisation weights.}\label{tab:incomp}
\footnotesize\setlength{\tabcolsep}{3pt}
\begin{tabular}{rrrr@{\quad}rrrr@{\quad}rrrr}
\toprule
& & & & \multicolumn{4}{c}{isotropic weight \eqref{eq:stab}} & \multicolumn{4}{c}{anisotropic weight}\\
\cmidrule(lr){5-8}\cmidrule(lr){9-12}
$k$ & $\nu$ & $\lambda_1$ & $\lambda\|\nabla\!\cdot\!\underline{u}\|_\infty$ & $E_e$ & rate & $E_0$ & rate & $E_e$ & rate & $E_0$ & rate\\
\midrule
1 & 0.3 & $5.77\cdot10^{-1}$ & 34.1 & $5.46\cdot10^{0}$ & 1.87 & $8.14\cdot10^{-2}$ & 2.97 & $8.45\cdot10^{0}$ & 1.87 & $1.72\cdot10^{-1}$ & 2.89\\
1 & 0.49 & $1.64\cdot10^{1}$ & 67.2 & $8.52\cdot10^{0}$ & 1.94 & $2.69\cdot10^{-2}$ & 3.33 & $8.54\cdot10^{0}$ & 1.91 & $1.59\cdot10^{-1}$ & 2.87\\
1 & 0.4999 & $1.67\cdot10^{3}$ & 69.5 & $1.00\cdot10^{2}$ & 1.97 & $1.89\cdot10^{-1}$ & 2.97 & $5.86\cdot10^{1}$ & 1.97 & $1.60\cdot10^{-1}$ & 2.88\\
1 & 0.499999 & $1.67\cdot10^{5}$ & 69.5 & $1.01\cdot10^{3}$ & 1.97 & $8.93\cdot10^{-1}$ & 3.09 & $5.83\cdot10^{2}$ & 1.97 & $1.60\cdot10^{-1}$ & 2.88\\
\midrule
2 & 0.3 & $5.77\cdot10^{-1}$ & 34.1 & $2.79\cdot10^{-1}$ & 2.97 & $3.45\cdot10^{-3}$ & 4.12 & $3.95\cdot10^{-1}$ & 2.94 & $6.44\cdot10^{-3}$ & 3.99\\
2 & 0.49 & $1.64\cdot10^{1}$ & 67.2 & $4.06\cdot10^{-1}$ & 2.97 & $1.32\cdot10^{-3}$ & 4.38 & $4.20\cdot10^{-1}$ & 2.97 & $5.38\cdot10^{-3}$ & 3.92\\
2 & 0.4999 & $1.67\cdot10^{3}$ & 69.5 & $5.09\cdot10^{0}$ & 3.01 & $1.67\cdot10^{-2}$ & 3.71 & $3.56\cdot10^{0}$ & 3.02 & $5.46\cdot10^{-3}$ & 3.95\\
2 & 0.499999 & $1.67\cdot10^{5}$ & 69.5 & $5.29\cdot10^{1}$ & 3.00 & $6.20\cdot10^{-2}$ & 4.46 & $3.55\cdot10^{1}$ & 3.02 & $5.46\cdot10^{-3}$ & 3.95\\
\midrule
3 & 0.3 & $5.77\cdot10^{-1}$ & 34.1 & $1.31\cdot10^{-2}$ & 3.88 & $1.83\cdot10^{-4}$ & 4.81 & $1.74\cdot10^{-2}$ & 3.89 & $2.57\cdot10^{-4}$ & 4.91\\
3 & 0.49 & $1.64\cdot10^{1}$ & 67.2 & $1.39\cdot10^{-2}$ & 3.97 & $5.28\cdot10^{-5}$ & 5.17 & $1.57\cdot10^{-2}$ & 3.95 & $1.59\cdot10^{-4}$ & 4.95\\
3 & 0.4999 & $1.67\cdot10^{3}$ & 69.5 & $1.77\cdot10^{-1}$ & 4.02 & $2.32\cdot10^{-4}$ & 5.60 & $1.27\cdot10^{-1}$ & 4.02 & $1.64\cdot10^{-4}$ & 4.96\\
3 & 0.499999 & $1.67\cdot10^{5}$ & 69.5 & $1.83\cdot10^{0}$ & 4.03 & $4.02\cdot10^{-3}$ & 4.89 & $1.26\cdot10^{0}$ & 4.02 & $1.64\cdot10^{-4}$ & 4.96\\
\bottomrule
\end{tabular}
\end{table}

\subsection{Test case 1 in the time domain}\label{ssec:case1dyn}

Since the exact stationary solution $\bu_s$ of \Cref{ssec:case1} satisfies homogeneous interface
conditions and the problem is linear, $\bu(\bx,t):=\chi(t)\bu_s(\bx)$ solves \eqref{eq:model}
with $\bff(t)=\chi(t)\bff_s+\chi''(t)\varrho\,\bu_s$, $\bu_0=\chi(0)\bu_s$,
$\bq_0=\chi'(0)\bu_s$ and time-dependent Dirichlet data on $\partial\Omega$, for any
$\chi\in C^2(\overline J)$. We take $\chi(t)=\cos(\omega t)$ with $\omega=2$, $\varrho=1$ and
$T=\tfrac12$. \Cref{tab:dynspace} reports the errors $E_e^\infty$ and $E_0^\infty$ obtained
with $\Delta t=1/2000$, which is small enough for the time-discretisation error to be
negligible: the rates $k+1$ and $k+2$ of \Cref{prop:energy,prop:L2} are recovered.
\Cref{tab:dyntime} reports, for $k=2$ and $n=8$, the difference at the final time between the
solutions computed with a given $\Delta t$ and with $\Delta t=1/3200$; the observed rates
approach the second order of the Newmark scheme.

\begin{table}[htbp]\centering
\begin{minipage}[t]{0.56\linewidth}\centering
\caption{Test case 1 in the time domain: errors and rates in space
($\Delta t=1/2000$, $T=\tfrac12$).}\label{tab:dynspace}
\small\setlength{\tabcolsep}{4pt}
\begin{tabular}{rrrrrrr}
\toprule
$k$ & $n$ & $N_{\rm dof}$ & $E_e^\infty$ & rate & $E_0^\infty$ & rate\\
\midrule
1 & 4 & 280 & $1.04\cdot10^{2}$ & -- & $9.34\cdot10^{0}$ & --\\
1 & 8 & 1200 & $3.41\cdot10^{1}$ & 1.61 & $1.47\cdot10^{0}$ & 2.67\\
1 & 16 & 4960 & $9.42\cdot10^{0}$ & 1.85 & $1.97\cdot10^{-1}$ & 2.90\\
1 & 32 & 20160 & $2.45\cdot10^{0}$ & 1.94 & $2.54\cdot10^{-2}$ & 2.96\\
\midrule
2 & 4 & 452 & $2.56\cdot10^{1}$ & -- & $2.06\cdot10^{0}$ & --\\
2 & 8 & 1928 & $3.57\cdot10^{0}$ & 2.84 & $1.34\cdot10^{-1}$ & 3.94\\
2 & 16 & 7952 & $4.59\cdot10^{-1}$ & 2.96 & $8.01\cdot10^{-3}$ & 4.07\\
2 & 32 & 32288 & $5.81\cdot10^{-2}$ & 2.98 & $4.89\cdot10^{-4}$ & 4.03\\
\midrule
3 & 4 & 656 & $4.67\cdot10^{0}$ & -- & $3.16\cdot10^{-1}$ & --\\
3 & 8 & 2784 & $3.30\cdot10^{-1}$ & 3.82 & $1.15\cdot10^{-2}$ & 4.79\\
3 & 16 & 11456 & $2.24\cdot10^{-2}$ & 3.88 & $4.18\cdot10^{-4}$ & 4.78\\
\bottomrule
\end{tabular}
\end{minipage}\hfill
\begin{minipage}[t]{0.4\linewidth}\centering
\caption{Test case 1 in the time domain: convergence in time ($k=2$, $n=8$).}
\label{tab:dyntime}
\begin{tabular}{rrr}
\toprule
$\Delta t$ & $\|\bm{u}_\mathcal{T}^N-\bm{u}^N_{\mathcal{T},\rm ref}\|_\Omega$ & rate\\
\midrule
$1/50$ & $1.82\cdot10^{-2}$ & --\\
$1/100$ & $5.48\cdot10^{-3}$ & 1.73\\
$1/200$ & $1.60\cdot10^{-3}$ & 1.77\\
$1/400$ & $4.92\cdot10^{-4}$ & 1.71\\
$1/800$ & $1.34\cdot10^{-4}$ & 1.87\\
\bottomrule
\end{tabular}
\end{minipage}
\end{table}

\subsection{Test case 2: wave propagation across a perfectly bonded unresolved interface}
\label{ssec:case2}

We now consider the propagation of an elastic wave in the two-dimensional heterogeneous
domain of \cite[Sect.~6.2]{BDES2021}, with the material interface rotated so that it is
\emph{not} resolved by the Cartesian background mesh. Let $\Omega=(-\tfrac32,\tfrac32)^2$ and
$\Gamma=\{\bx\mid\bn\cdot\bx=0\}$ with
$\bn=(-\sin\varphi,\cos\varphi)^{\mathsf T}$, $\varphi=\pi/12$; $\Omega_1$ is the lower
subdomain. The material properties are $\varrho_1=\varrho_2=1$,
the speed of S-waves $c_{{\rm S},1}=c_{{\rm S},2}/2=1$, the speed of P-waves $c_{{\rm P},1}=c_{{\rm P},2}/2=\sqrt3$, i.e.\
$\mu_1=\lambda_1=1$ and $\mu_2=\lambda_2=4$. The simulation time is $T=1$, the source term is
$\bff=\bm{0}$, homogeneous Dirichlet conditions are enforced on $\partial\Omega$, and the
initial conditions are $\bu_0=\bm{0}$ and
\begin{equation}\label{eq:ricker}
\bq_0(\bx)=\theta\exp\Big(-\pi^2\frac{r^2}{\lambda_w^2}\Big)\,(x_1-x_c,x_2-y_c)^{\mathsf T},
\qquad r^2=(x_1-x_c)^2+(x_2-y_c)^2,
\end{equation}
with $\theta=10^{-2}$, $\lambda_w=c_{{\rm P},2}/f_c$, $f_c=5$ and
$(x_c,y_c)=(0,\tfrac23)$, i.e.\ a Ricker wave centred in $\Omega_2$. Compared with
\cite{BDES2021} we halve the central frequency from $f_c=10$ to $f_c=5$, so that the pulse is
adequately resolved on the meshes that are within reach of our verification code: the number
of points per wavelength is $\lambda_w/h=7.4$ for $n=32$ and $14.8$ for $n=64$. The
wave first propagates in $\Omega_2$, is then partially transmitted to $\Omega_1$ and is later
reflected at the boundary of $\Omega$. In this test case the interface is perfectly bonded,
$\bK= \mathsf{0}$, so that the method reduces to the Nitsche-like regime of \Cref{rem:regimes};
the discrete jump is then penalised with the weight $\delta h_T^{-1}$.

As in \cite{BDES2021} we track the two Cartesian components of the velocity at two sensors,
$S_1=(\tfrac13,-\tfrac13)\in\Omega_1$ and $S_2=(\tfrac13,\tfrac13)\in\Omega_2$. Since no
analytical solution is available for this configuration, we use as a reference the solution
computed with $k=4$ on the finest mesh. All the computations reported in this subsection and
in \Cref{ssec:case3} use the SDIRK$(3,4)$ scheme with $\Delta t=1/640$, which is small enough
for the temporal error to be negligible; this is quantified below, where the choice of the
time-stepping scheme is discussed. \Cref{tab:case2} reports the maximum relative error, in
percent, over the discrete time nodes of $[0,t_{*i}]$, normalised by the maximum value in
time of the corresponding velocity component of the reference solution; as in
\cite{BDES2021} we set $t_{*1}=0.9$ and $t_{*2}=0.6$, which are the times at which the waves
reflected at $\partial\Omega$ reach the sensors $S_1$ and $S_2$. Meshes with fewer than about
seven points per wavelength are not reported, since they do not resolve the pulse: for
$n=16$, i.e.\ $\lambda_w/h=3.7$, the errors exceed $70\%$ for all three polynomial degrees.

The convergence with respect to both the mesh size and the polynomial degree is clear. At
$n=64$ the error on the most demanding component falls from $16\%$ for $k=1$ to $2.5\%$ for
$k=2$ and $0.09\%$ for $k=3$; the finest configuration reproduces the reference to better
than one tenth of a percent on all four components, although no mesh face is aligned with
$\Gamma$ and every cut cell of the mesh carries doubled unknowns. The error is largest for
$v_x$ at $S_1$, which is the smallest component, its amplitude being one order of magnitude
below that of the three others (see \Cref{fig:traces2}); the measure is deliberately severe,
being the maximum over the whole time window of a \emph{pointwise} velocity component,
normalised by the maximum of that component alone.

The choice of the time-stepping scheme deserves a comment. \Cref{tab:dtstudy} reports the
same errors for the fixed space discretisation $k=3$, $n=32$, as a function of $\Delta t$,
for the Newmark scheme and for SDIRK$(3,4)$. At $\Delta t=1/640$ the two schemes agree to
within half a percentage point, which confirms that the entries of \Cref{tab:case2} measure
the space discretisation alone. When $\Delta t$ is increased the Newmark errors deteriorate
rapidly --- by a factor of six on $v_y$ at $S_2$ between $\Delta t=1/640$ and
$\Delta t=1/80$ --- whereas SDIRK$(3,4)$ has already reached the level of the space error at
$\Delta t=1/160$, a time step four times larger. Since one SDIRK$(3,4)$ step costs three
Newmark steps, the net saving is a factor $4/3$ here, and it grows with the accuracy required
of the space discretisation; this, together with the mild and controlled dissipation reported
in the last column, is why we use SDIRK$(3,4)$ throughout. \Cref{tab:timeconv} confirms the expected temporal orders $2$,
$3$ and $4$ of the three schemes of \Cref{sec:time} on the same configuration; the asymptotic
rates are reached once $\Delta t$ resolves the fastest modes of the discrete operator.

The velocity traces are displayed in \Cref{fig:traces2} and snapshots of the two velocity
components at four times are shown in \Cref{fig:snaps2}; the transmitted and reflected wave
fronts cross the interface without any visible spurious reflection. Finally,
\Cref{fig:energy2} illustrates the energy behaviour of the two schemes: the Newmark scheme
conserves \eqref{eq:discenergy} to machine precision, as predicted by
\eqref{eq:energybalance}, whereas SDIRK$(3,4)$ dissipates a small amount of energy, about
$2\cdot10^{-5}$ of the initial energy at $\Delta t=1/640$ and decreasing by a factor of
roughly $30$ at each halving of $\Delta t$,
in agreement with \Cref{rem:rkenergy}.

\begin{table}[htbp]\centering
\caption{Test case 2. Maximum relative error (in \%) of the velocity components at the two
sensors over $[0,t_{*i}]$, with respect to the reference solution ($k=4$, $n=64$,
$N_{\rm dof}=256\,584$); SDIRK$(3,4)$ with $\Delta t=1/640$.}\label{tab:case2}
\begin{tabular}{rrrrrrrr}
\toprule
$k$ & $n$ & $\lambda_w/h$ & $N_{\rm dof}$ & $v_x(S_1)$ & $v_y(S_1)$ & $v_x(S_2)$ & $v_y(S_2)$\\
\midrule
1 & 32 & 7.4 & 20928 & 80.93 & 58.01 & 16.95 & 34.12\\
1 & 64 & 14.8 & 82848 & 16.05 & 12.02 & 1.69 & 2.89\\
\midrule
2 & 32 & 7.4 & 33464 & 17.66 & 22.10 & 2.83 & 9.16\\
2 & 64 & 14.8 & 132516 & 2.51 & 0.92 & 0.26 & 0.17\\
\midrule
3 & 32 & 7.4 & 48072 & 8.08 & 4.05 & 1.25 & 1.27\\
3 & 64 & 14.8 & 190428 & 0.09 & 0.09 & 0.02 & 0.01\\
\bottomrule
\end{tabular}
\end{table}

\begin{figure}[htbp]\centering
\includegraphics[width=\linewidth]{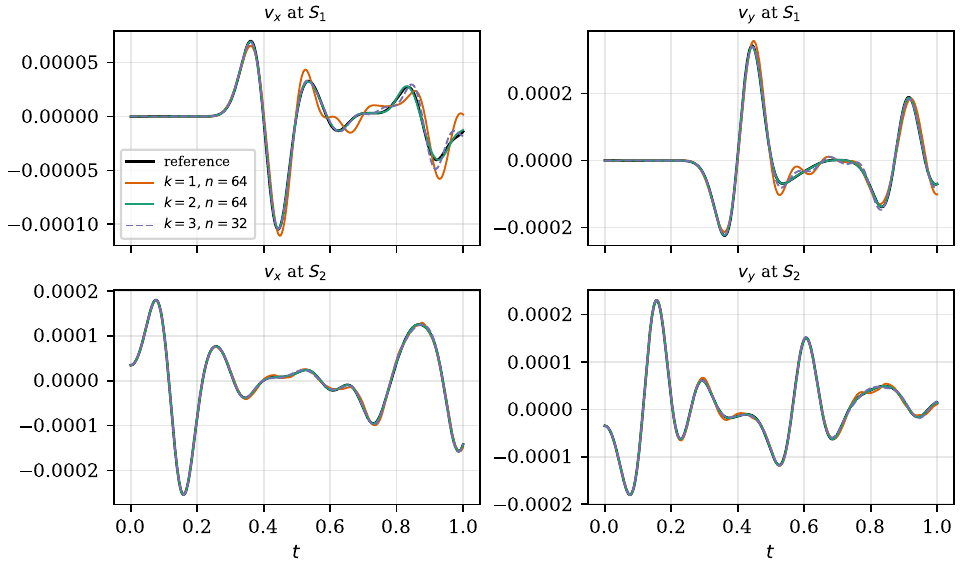}
\caption{Test case 2. Velocity components at the two sensors as a function of
time.}\label{fig:traces2}
\end{figure}

\begin{figure}[htbp]\centering
\includegraphics[width=\linewidth]{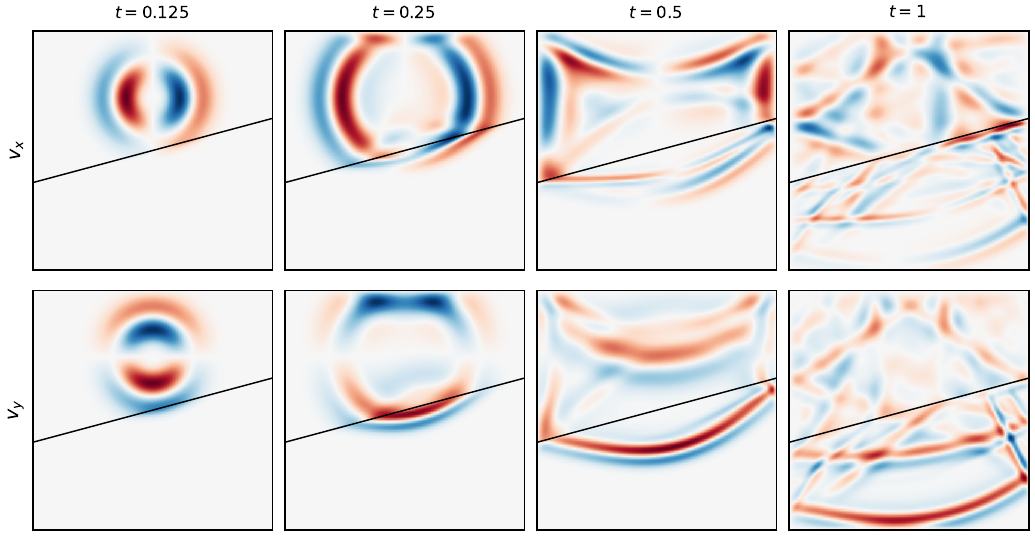}
\caption{Test case 2. Snapshots of $v_x$ (top) and $v_y$ (bottom) at
$t\in\{\tfrac18,\tfrac14,\tfrac12,1\}$; the interface is drawn in black
($k=3$, $n=64$).}\label{fig:snaps2}
\end{figure}

\begin{figure}[htbp]\centering
\includegraphics[width=\linewidth]{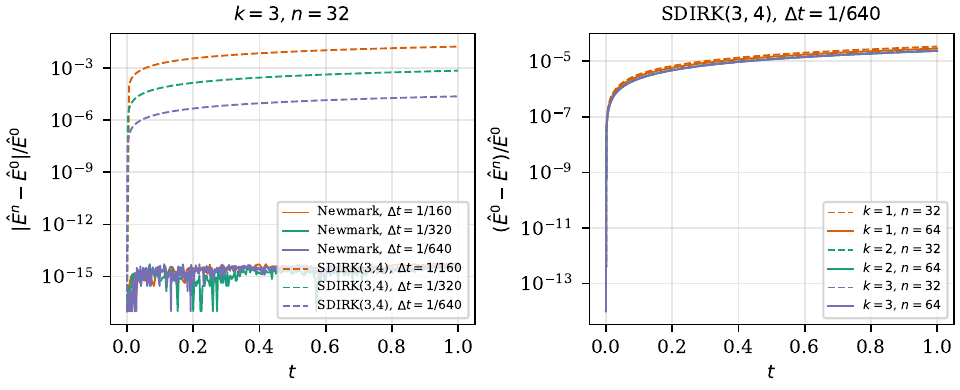}
\caption{Test case 2. Relative variation of the discrete energy \eqref{eq:discenergy}: the
Newmark scheme (left) conserves it to machine precision, the SDIRK schemes (right) dissipate
it.}\label{fig:energy2}
\end{figure}

\begin{table}[htbp]\centering
\caption{Test case 2 with $k=3$, $n=32$: maximum relative error (in \%) at the two sensors
over $[0,t_{*i}]$ and relative energy loss over $[0,T]$, as a function of the time step, for
the Newmark and the SDIRK$(3,4)$ schemes.}\label{tab:dtstudy}
\small
\begin{tabular}{llrrrrr}
\toprule
scheme & $\Delta t$ & $v_x(S_1)$ & $v_y(S_1)$ & $v_x(S_2)$ & $v_y(S_2)$ & energy loss\\
\midrule
Newmark & $1/80$ & 29.01 & 28.97 & 10.56 & 26.40 & $<10^{-12}$\\
 & $1/160$ & 12.32 & 10.46 & 2.61 & 7.92 & $<10^{-12}$\\
 & $1/320$ & 8.95 & 5.60 & 1.13 & 2.43 & $<10^{-12}$\\
 & $1/640$ & 8.28 & 4.41 & 1.22 & 1.43 & $<10^{-12}$\\
\midrule
SDIRK(3,4) & $1/80$ & 10.33 & 10.56 & 3.74 & 8.45 & $1.9\cdot10^{-1}$\\
 & $1/160$ & 6.24 & 3.71 & 1.27 & 0.96 & $1.7\cdot10^{-2}$\\
 & $1/320$ & 7.99 & 4.01 & 1.23 & 1.23 & $7.1\cdot10^{-4}$\\
 & $1/640$ & 8.08 & 4.05 & 1.25 & 1.27 & $2.4\cdot10^{-5}$\\
\bottomrule
\end{tabular}
\end{table}

\begin{table}[htbp]\centering
\caption{Test case 2 with $k=2$, $n=16$ and $T=\tfrac14$: difference
$\|\bm{u}^N_\mathcal{T}-\bm{u}^N_{\mathcal{T},\rm ref}\|_\Omega$ with the solution computed
with SDIRK$(3,4)$ and $\Delta t=1/5120$, and observed temporal rates.}\label{tab:timeconv}
\small
\begin{tabular}{rrrrrrr}
\toprule
& \multicolumn{2}{c}{Newmark} & \multicolumn{2}{c}{SDIRK(2,3)} & \multicolumn{2}{c}{SDIRK(3,4)}\\
\cmidrule(lr){2-3}\cmidrule(lr){4-5}\cmidrule(lr){6-7}
$\Delta t$ & err & rate & err & rate & err & rate\\
\midrule
$1/40$ & $3.93\cdot10^{-6}$ & -- & $2.48\cdot10^{-6}$ & -- & $2.14\cdot10^{-6}$ & --\\
$1/80$ & $1.15\cdot10^{-6}$ & 1.77 & $5.47\cdot10^{-7}$ & 2.18 & $4.16\cdot10^{-7}$ & 2.36\\
$1/160$ & $3.02\cdot10^{-7}$ & 1.94 & $8.84\cdot10^{-8}$ & 2.63 & $4.67\cdot10^{-8}$ & 3.15\\
$1/320$ & $7.61\cdot10^{-8}$ & 1.99 & $1.19\cdot10^{-8}$ & 2.89 & $3.59\cdot10^{-9}$ & 3.70\\
$1/640$ & $1.91\cdot10^{-8}$ & 2.00 & $1.51\cdot10^{-9}$ & 2.98 & $2.45\cdot10^{-10}$ & 3.87\\
\bottomrule
\end{tabular}
\end{table}

\subsection{Test case 3: wave propagation across a linear slip interface}\label{ssec:case3}

We finally repeat the experiment of \Cref{ssec:case2} with a compliant interface,
$\alpha=\beta\in\{0,5\cdot10^{-3},2\cdot10^{-2},5\cdot10^{-2}\}$, all the other data being
unchanged. The relevant dimensionless parameters are $\omega_c$ and $Z_1$, where
$\omega_c=2\pi f_c$ and $Z_1=\varrho_1c_{{\rm P},1}$ is the acoustic impedance of the softer
medium; it ranges from $0$ to about $8$, so that the interface goes from perfectly bonded to
almost traction free. \Cref{fig:case3} (left) shows the magnitude of the discrete slip
$\jump{\bu_h}$ along $\Gamma$ at the final time --- recall that the jump is not an unknown of
the method and is recovered \emph{a posteriori} from the two cell unknowns of the cut cells
--- which grows with the compliancy, as expected. The centre and right panels show the
velocity traces at the two sensors: increasing the compliancy reduces the amplitude
transmitted to $\Omega_1$ and increases the amplitude reflected back into $\Omega_2$, which
is the physically expected behaviour of a linear slip interface \cite{Schoenberg1980}. The
same effect is visible in \Cref{fig:case3snap}, which compares the vertical velocity at
$t=\tfrac14$, when the incident wave front sweeps the interface: for the perfectly bonded
interface the transmitted front in $\Omega_1$ is clearly formed, whereas for
$\alpha=\beta=5\cdot10^{-2}$ it has almost disappeared and the energy is reflected back into
$\Omega_2$.
\Cref{tab:case3} reports, for $\alpha=\beta=2\cdot10^{-2}$, the same maximum relative errors
as in \Cref{tab:case2}, computed with respect to the solution obtained with $k=4$ on the
finest mesh; the finest configuration reproduces the reference to within $0.15\%$ on all four
components, an accuracy comparable to the perfectly bonded case, which confirms that the
method is not degraded by the presence of the compliant interface.

\begin{figure}[htbp]\centering
\includegraphics[width=\linewidth]{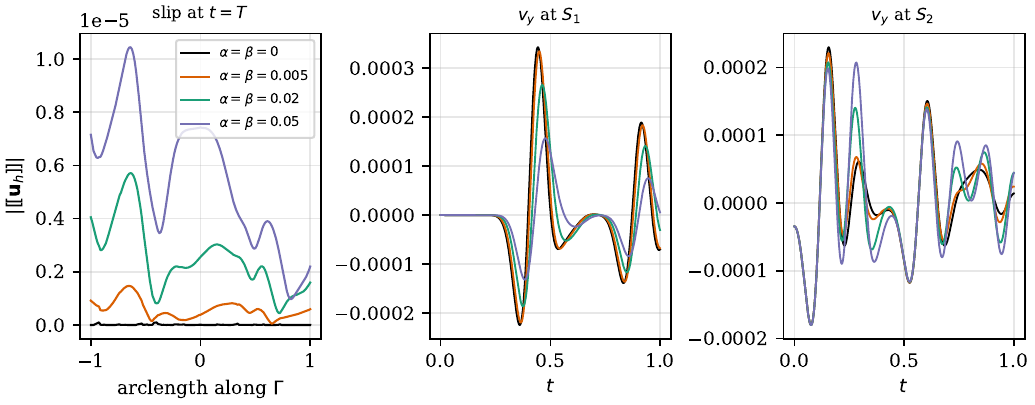}
\caption{Test case 3. Left: magnitude of the discrete slip $\jump{\bu_h}$ along $\Gamma$ at
$t=T$ for several compliancies. Centre and right: $v_y$ at the two sensors ($k=3$,
$n=64$).}\label{fig:case3}
\end{figure}

\begin{figure}[htbp]\centering
\includegraphics[width=\linewidth]{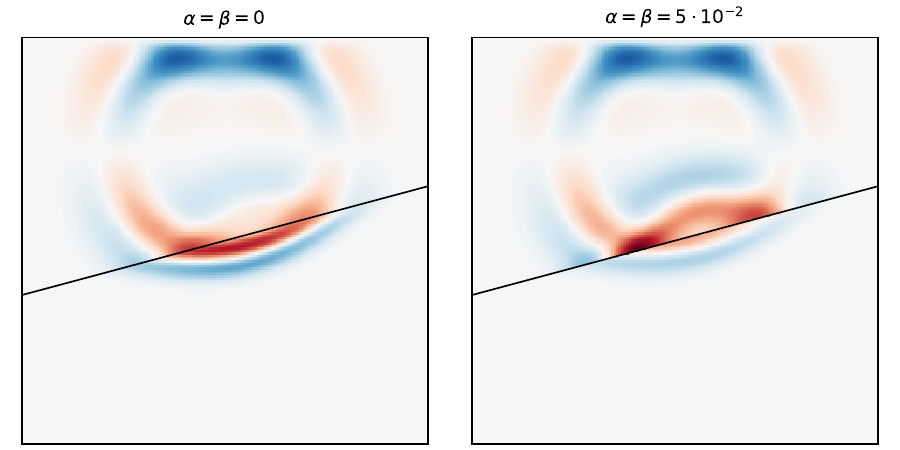}
\caption{Test case 3. Vertical velocity $v_y$ at $t=\tfrac14$ for a perfectly bonded (left)
and for a compliant (right) interface ($k=3$, $n=64$); the colour scale is the same in both
panels.}\label{fig:case3snap}
\end{figure}

\begin{table}[htbp]\centering
\caption{Test case 3 ($\alpha=\beta=2\cdot10^{-2}$). Maximum relative error (in \%) of the
velocity components at the two sensors over $[0,t_{*i}]$, with respect to the reference
solution ($k=4$, $n=64$); SDIRK$(3,4)$ with $\Delta t=1/640$.}\label{tab:case3}
\begin{tabular}{rrrrrrrr}
\toprule
$k$ & $n$ & $\lambda_w/h$ & $N_{\rm dof}$ & $v_x(S_1)$ & $v_y(S_1)$ & $v_x(S_2)$ & $v_y(S_2)$\\
\midrule
1 & 32 & 7.4 & 20928 & 51.08 & 54.10 & 9.26 & 19.69\\
1 & 64 & 14.8 & 82848 & 11.91 & 11.03 & 1.10 & 2.84\\
\midrule
2 & 32 & 7.4 & 33464 & 17.77 & 18.80 & 1.57 & 5.14\\
2 & 64 & 14.8 & 132516 & 1.85 & 0.80 & 0.18 & 0.07\\
\midrule
3 & 32 & 7.4 & 48072 & 6.39 & 2.84 & 0.88 & 0.73\\
3 & 64 & 14.8 & 190428 & 0.15 & 0.07 & 0.01 & 0.01\\
\bottomrule
\end{tabular}
\end{table}

\section{Conclusions}\label{sec:concl}

We have designed and analysed an unfitted hybrid high-order method for the elastodynamics
problem with an imperfect interface of linear slip type. The method uses polynomials of
degree $k\ge1$ on the mesh faces and of degree $k+1$ in the mesh cells, the unknowns being
doubled in the cells cut by the interface, and it supports general polytopal cells and hanging
nodes, so that the cells carrying small cuts can be cured by agglomeration. Its two specific
ingredients are a local symmetric strain reconstruction on the softer side of each cut cell,
into which the interface condition is built through the factor $\bI-\bS_h\bK$, and an
interface stabilisation weighted by the regularised interface stiffness $\bS_h$ of
\eqref{eq:Sh}. As a consequence, one single formulation covers the whole range of
compliancies, from the perfectly bonded to the traction-free interface, and no unknown is
attached to the interface. We have proved that the discrete bilinear form is coercive and
continuous on the discrete space, with constants independent of the compliancy and of how the
interface cuts the mesh, and that the space semi-discrete errors converge as $h^{k+1}$ in the
energy norm and as $h^{k+2}$ in the $L^2$-norm. Combined with the Newmark scheme, the method
conserves a discrete energy exactly; combined with singly diagonally implicit Runge--Kutta
schemes of order up to four, it dissipates a controlled amount of energy and tolerates time
steps several times larger for the same accuracy. The numerical experiments of \Cref{sec:num} confirm these
rates for $k\in\{1,2,3\}$, the robustness of the method with respect to the compliancy over
sixteen orders of magnitude, including the degenerate cases $\alpha=0$ or $\beta=0$, the
necessity and the effectiveness of the agglomeration procedure in the presence of small cuts,
and the ability of the method to propagate elastic waves across an unresolved compliant
interface.



\paragraph{Funding} This work was supported by the National Natural Science Foundation of China grant 11301267, the Natural Science Foundation of Jiangsu Province grant BK20191386 and the Qing Lan Project of Jiangsu
Province. EB was partially supported by EPSRC grants EP/P01576X/1 and EP/V050400/1.

\paragraph{Data availability}
The code reproducing all numerical examples of this paper is openly available at \cite{Huang_v1_0_0_Reproduction_bundle}:\\[3mm]
{\tt{https://github.com/burmanerik/unfitted-hho-elastodynamics}}\\[3mm]
 and archived at:\\[3mm]
{\tt{https://doi.org/10.5281/zenodo.23081108.}}\\[3mm]
The repository contains the solver, the drivers for the three test cases, and the raw output of the production runs, so that every table and figure of Section 7 can be regenerated in about a minute without repeating the computations; it also includes a test suite verifying the convergence rates and the structural properties of the method.


\paragraph{Acknowledgement}
During the preparation of this manuscript, the authors used Anthropic Claude Opus 5, to assist with drafting and revising text, improving mathematical exposition, checking notation and internal consistency, editing LaTeX, and developing computational code. These tools were not used to fabricate or directly alter research data or numerical results. All AI-assisted material, including mathematical statements and computational code, was critically reviewed and validated by the authors, who take full responsibility for the accuracy, originality, and integrity of the manuscript.
\bibliographystyle{plain}
\bibliography{ref1}
\end{document}